\documentclass[11pt]{article}
\usepackage{amsmath,amsfonts,amssymb,amsthm,enumerate,graphicx}
\usepackage{tikz,authblk}
\usepackage{tikz-cd}
\usepackage{subfig}
\usepackage{url}
\usepackage{float}
\usepackage{easyReview}

\usepackage{hyperref}
\hypersetup{
	colorlinks=true,
	citecolor=black,
	linkcolor=black,
	filecolor=black,      
	urlcolor=black,}

\newtheorem{theorem}{{\bf Theorem}}[section]
\newtheorem{result}{{\bf Result}}[section]
\newtheorem{claim}{{\bf Claim}}[section]

\newtheorem{example}[theorem]{{\bf Example}}
\newtheorem{lemma}[theorem]{{\bf Lemma}}

\newtheorem{proposition}[theorem]{{\bf Proposition}}

\newtheorem{question}[theorem]{{\bf Question}}

\newcommand{\Kc}{\mathcal{K}}
\newcommand{\Hc}{\mathcal{H}}

\newcommand{\ZZ}{ \ensuremath{\mathbb{Z}}}

\begin{document}

\title{Semi\mbox{-}equivelar toroidal maps and their vertex covers}

	\author[1] {Arnab Kundu}
 	\author[2] { Dipendu Maity}
 	\affil[1, 2]{Department of Science and Mathematics,
 		Indian Institute of Information Technology Guwahati, Bongora, Assam-781\,015, India.\linebreak
 		\{arnab.kundu, dipendu\}@iiitg.ac.in/\{kunduarnab1998, dipendumaity\}@gmail.com.}

\date{\today}

\maketitle

\begin{abstract}
%A map $\mathcal{K}$ on a surface is called vertex-transitive if the automorphism group of $\mathcal{K}$ acts transitively on the set of vertices of $\mathcal{K}$. If the face\mbox{-}cycles at all the vertices in a map are of same type then the map is called semi\mbox{-}equivelar. In general, semi\mbox{-}equivelar maps on a surface form a bigger class than vertex-transitive maps. There are semi\mbox{-}equivelar toroidal maps which are not vertex\mbox{-}transitive. A map is called minimal if the number of vertices is minimal. A map $f \colon \mathcal{M} \to \mathcal{K}$ is a covering map if $f$ is a covering map from the vertex set of $\mathcal{M}$ to the vertex set of $\mathcal{K}$ such that $f$ is a surjection and a local isomorphism $\colon$ the neighbourhood of a vertex $v$ in $\mathcal{M}$ is mapped bijectively onto the neighbourhood of $f(v)$ in $\mathcal{K}$. A map is called $k$-vertex orbital or $k$-orbital if it contains $k$ number of vertex orbits. In particular, if $k=1$ then it is called vertex-transitive map.

If the face\mbox{-}cycles at all the vertices in a map are of same type then the map is called semi\mbox{-}equivelar. A map is called minimal if the number of vertices is minimal. We know the bounds of number of vertex orbits of semi-equivelar toroidal maps. These bounds are sharp. Datta \cite{BD2020} has proved that every semi-equivelar toroidal map has a vertex-transitive cover. In this article, we prove that if a semi-equivelar map is $k$ orbital then it has a finite index $m$-orbital minimal cover for $m \le k$. We also show the existence and classification of $n$-sheeted covers of semi-equivelar toroidal maps for each $n \in \mathbb{N}$. 
\end{abstract}

\noindent {\small {\em MSC 2010\,:} 52C20, 52B70, 51M20, 57M60.

 }

\section{Introduction}

A {\em map M} is an embedding of a graph {\em G} on a surface {\em S} such that the closure of components of $S \setminus G$, called the {\em faces} of $M$, are homeomorphic to $2$-discs. A map $M$ is said to be a {\em polyhedral map} if the intersection of any two distinct faces is either empty, a common vertex, or a common edge. Here map means a polyhedral map.

The $face\mbox{-}cycle$ $C_u$ of a vertex $u$ (also called the {\em vertex-figure} at $u$) in a map is the ordered sequence of faces incident to $u$.
So, $C_u$ is of the form $(F_{1,1}\mbox{-}\cdots \mbox{-}F_{1,n_1})\mbox{-}\cdots\mbox{-}(F_{k,1}\mbox{-}$ $\cdots \mbox{-}F_{k,n_k})\mbox{-}F_{1,1}$, where $F_{i,\ell}$ is a regular $p_i$-gon for $1\leq \ell \leq n_i$, $1\leq i \leq k$, $p_r\neq p_{r+1}$ for $1\leq r\leq k-1$ and $p_n\neq p_1$. The types of the faces in $C_u$ defines the type of $C_u$. In this case, the type of face-cycle($u$) is $[p_1^{n_1}, \dots, p_k^{n_k}]$, is called vertex type of $u$. A map $M$ is called {\em semi-equivelar} (\cite{DM2018}, we are including the same definition for the sake of completeness) if $C_u$ and $C_v$ are of same type for all $u, v \in V(X)$. More precisely, there exist integers $p_1, \dots, p_k\geq 3$ and $n_1, \dots, n_k\geq 1$, $p_i\neq p_{i+1}$ (addition in the suffix is modulo $k$) such that $C_u$ is of the form as above for all $u\in V(X)$. In such a case, $X$ is called a semi-equivelar map of type (or vertex type) $[p_1^{n_1}, \dots, p_k^{n_k}]$ (or, a map of type $[p_1^{n_1}, \dots, p_k^{n_k}]$). 

Two maps of fixed type on the torus are {\em isomorphic} if there exists a {\em homeomorphism} of the torus which maps vertices to vertices, edges to edges, faces to faces and preserves incidents. More precisely,
if we consider two polyhedral complexes $M_{1}$ and $M_{2}$ then an isomorphism to be a map $f ~:~ M_{1}\rightarrow M_{2}$ such that $f|_{V(M_{1})} : V(M_{1}) \rightarrow V(M_{2})$ is a bijection and $f(\sigma)$ is a cell in $M_{2}$ if and only if $\sigma$ is a cell in $M_{1}$. In particular, if $M_1 = M_2$, then $f$ is called an $automorphism$. The \emph{automorphism group $Aut(M)$} of $M$ is the group consisting of automorphisms of $M$. A toroidal map $M$ is called {\em $k$-orbital} if it has $k$ number of vertex orbits under the action of automorphism group Aut($M$) on the set of vertices of $M$.

When working with discrete symmetric structures on a torus, many of the ideas follow the concepts introduced by Coxeter and Moser in \cite{CM1957}.  A surjective mapping $\eta \colon X \to Y$ from a map $X$ to a map $Y$ is called a $covering$ if it preserves adjacency and sends vertices, edges, faces of $X$ to vertices, edges, faces of $Y$ respectively. (In particular, $\eta$ is called $n$-sheeted coveing map if the cardinality of $\eta^{-1}(v)$ is $n$ for every $v \in V(Y)$.) That is, let $G \leq$Aut($X$) be a discrete group acting on a map $X$ \emph{properly discontinuously} (\cite[Chapter 2]{katok:1992}). This means that each element $g$ of $G$ is associated to a automorphism $h_g$ of $X$ onto itself, in such a way that $h_{gh}$ is always equal to $h_g h_h$ for any two elements $g$ and $h$ of $G$, and $G$-orbit of any vertex $u\in V(X)$ is locally finite. Then, there exists $\Gamma \leq $Aut($X$) such that $Y = X/\Gamma$. In such a case, $X$ is called a cover of $Y$. A map $X$ is called regular if the automorphism group of $X$ acts transitively on the set of flags of $X$. Clearly, if a semi-equivelar map is not equivelar then it cannot be regular. A semi-equivelar map $X$ on the torus is called almost regular if it has the same number of flag orbits under the action of its automorphism group as the Archimedean tiling $X$ on the plane of the same type has under the action of its symmetry group. In \cite{drach:2019}, the authors consider vertex-transitive maps on the torus and proved that each vertex-transitive map on the torus has a minimal almost regular cover (on the torus). They also showed that such a minimal almost regular cover is unique and can be constructed explicitly. There is also much interest in finding minimal regular covers of different families of
maps and polytopes (see \cite{HW2012, MPW2013, pw2011}). 

\smallskip

A natural question then is: 

\begin{question}\label{ques}
Let $X$ be a semi-equivelar  map on the torus. Does there exist any map $Y\neq X$ which is a cover of $X$ and has less number of orbit under the action of automorphism group? Does this cover exist for every sheet, if so, how many? 
%How the orbits of $X$ and $Y$ are related? 
\end{question}

In this context, we know the following on the torus.

\begin{proposition} (\cite{BD2020}) \label{datta2020}
If $X$ is a semi\mbox{-}equivelar toroidal map then there exists a covering $\eta \colon Y \to X$ where $Y$ is a vertex\mbox{-}transitive toroidal map.
\end{proposition}

Here we prove the following. 

\begin{theorem}\label{thm-main0}
 There does not exist $4$ and $5$-orbital map of type $[4^1,6^1,12^1]$ and $2$-orbital map of type $[3^4,6^1]$ on the torus.
\end{theorem}

\begin{theorem}\label{thm-main1}
%{\rm (a)} If $X_1$ is a $m_1$-orbital toroidal map of type $[3^2,4^1,3^1,4^1]$ or $[4^1,8^2]$, then there exists a covering $\eta_{\mathcal{K}_3} \colon Y_{\mathcal{K}_3} \to X_1$ where $Y_{\mathcal{K}_3}$ is $\mathcal{K}_3$-orbital for each $\mathcal{K}_3 \le m_1$.
{\rm (a)} If $X_2$ is a $m_2$-orbital toroidal map of type $[3^1,6^1,3^1,6^1]$, $[3^1,4^1,6^1,$ $4^1]$ or $[3^1, 12^2]$, then there exists a covering $\eta_{k_2} \colon Y_{k_2} \to X_2$ where $Y_{k_2}$ is $k_2$-orbital for each $k_2 \le m_2$.
{\rm (b)} If $X_3$ is a $m_3$-orbital toroidal map of type $[4^1,6^1,12^1]$, then there exists a covering $\eta_{k_3} \colon Y_{k_3} \to X_3$ where $Y_{k_3}$ is $k_3$-orbital for each $k_3 (\neq 4,5) \le m_3$ except $(k_3, m_3) =  (2,3)$.
\end{theorem}

\begin{theorem}\label{thm-main2}
Let $X$ be a semi\mbox{-}equivelar toroidal map. Then, there exists a $n$-sheeted covering $\eta \colon Y \to X$ for each $n \in \mathbb{N}$.
\end{theorem}

\begin{theorem}\label{thm-main3}
Let $X$ be a $n$-sheeted semi\mbox{-}equivelar toroidal map. Then, there exists different $n$-sheeted  coverings $\eta_{\ell} \colon Y_{\ell} \to X$ for $ \ell \in \{1, 2, \dots, \Lambda(n)\}$, i.e., $Y_{1},$ $Y_{2},$ $\dots,$
$Y_{\Lambda(n)}$ are $n$-sheeted covers of $X$ and different up to isomorphism where
$\Lambda(n) = \frac{1}{4}[\sigma(n)+f_2(n)+g(n)+f_3(n)]$ for type $[3^2,4^1,3^1,4^1]$; 
$\Lambda(n) = \frac{1}{6}(\sigma(n)+2f_1(n)+3f_3(n))$ for types $[3^1,6^1,3^1,6^1]$, $[3^1,12^2]$, $[3^1,4^1,6^1,4^1]$, $[4^1,6^1,12^1]$;
$\Lambda(n)= \frac{\sigma(n)+2f_1(n)}{3}$ for type $[3^4,6^1]$ and 
$\Lambda(n) = \frac{1}{4}[\sigma(n)+2f_2(n)+4f_3(n)-3f_6(n)]$ for type $[4^1,8^2]$,
      where for $n \in \mathbb{N}$,
      
      \begin{itemize}
      \item[1.] $\sigma (n):= \sum_{d\mid n}d$,
      
     \item[2.]$f_1(n) := \left\{
	\begin{array}{ll}
		0  & \mbox{, if } m_j\equiv1\pmod2 \mbox{ for some } j\in\{0,1,2,\dots,n_2\} \\
		\prod_{i=1}^{n_1}(k_i+1) & \mbox{, otherwise,}
	\end{array}
      \right. $
      
      where $n=2^{m_0} \cdot 3^{k_0} \cdot \prod_{i=1}^{n_1}p_i^{k_i} \cdot \prod_{j=1}^{n_2}q_j^{m_j}$ with $p_i$ and $q_j$ are primes such that $p_i\equiv1\pmod3$ for $i\in \{0,1,\dots,n_1\}$ and $q_j\equiv2\pmod3$ for $j\in \{0,1,\dots,n_2\}$, 
      
    \item[3.] $ f_2(n):=\left\{
	\begin{array}{ll}
		0  & \mbox{, if } m_j\equiv1\pmod2 \mbox{ for some } j\in\{0,1,2,\dots,n_2\} \\
		\prod_{i=1}^{n_1}(k_i+1) & \mbox{, otherwise.}
	\end{array}
        \right. $
        
        where $n=2^{m_0} \cdot 3^{k_0} \cdot \prod_{i=1}^{n_1}p_i^{k_i} \cdot \prod_{j=1}^{n_2}q_j^{m_j}$ such that $p_i$ and $q_j$ are primes with $p_i\equiv1\pmod4$ for $i\in \{0,1,\dots,n_1\}$ and $q_j\equiv3\pmod4$ for $j\in \{0,1,\dots,n_2\}$,

     \item[4.] $f_3(n):=\left\{
	 \begin{array}{ll}
		\prod_{i=1}^{n_1}(k_i+1)  & \mbox{, if } k_0=0 \\
		(2k_0-1)\prod_{i=1}^{n_1}(k_i+1) & \mbox{, otherwise.}
	 \end{array}
     \right.$
     
     where $n=2^{k_0}\cdot \prod_{i=1}^{n_1}p_i^{k_i}$ such that $p_i$ is any odd prime for $i\in \{0,1,\dots,n_1\}$,

    \item[6.]$f_6(n):=\left\{
	\begin{array}{ll}
		0  & \mbox{, if } k_i\equiv1\pmod2 \mbox{ for some } i\in\{0,1,2,\dots,n_1\}\\
		1 & \mbox{, otherwise.}
	\end{array}
    \right.$
    
    where 
    $n=2^{m}\cdot 3^{k_0} \cdot \prod_{i=1}^{n_1}p_i^{k_i}$ such that $p_i$ is any prime other than $2$ for $i\in \{0,1,\dots,n_1\}$.
    % and $f_7(n) = f_3(n)$.
    
    \item[7.] $g(n) := \sum_{\substack{d \mid n \\2 \mid d}}2 + \sum_{\substack{d \mid n \\2 \nmid d}}1.$
    
   % \item[] $\alpha(n) = f_3(n)-h(n)$.
    
\end{itemize}
\end{theorem}

\begin{theorem}\label{thm-main4}
Let $X$ be a $m$-orbital semi\mbox{-}equivelar toroidal map and $Y$ be a $k$-orbital covers of $X$. Then, there exists a $k$-orbital covering map $\eta \colon Z \to X$ such that $Z$ is minimal.
\end{theorem}
 
\section{Definitions, results, and ideas of the proofs of the theorems}

Throughout the last few decades there have been many results about maps and semi-equivelar maps that are highly symmetric. In particular, there has been
recent interest in the study of discrete objects using combinatorial, geometric, and algebraic approaches, with the topic of symmetries of maps receiving a lot of interest. There is a great history of work surrounding maps on the Euclidean plane $\mathbb{R}^2$ and  the $2$-dimensional torus.

An {\em Archimedean} tiling of the plane $\mathbb{R}^2$ is a tiling of $\mathbb{R}^2$ by regular polygons such that all the vertices of the tiling are of same type.
Gr\"{u}nbaum and Shephard \cite{GS1977} showed that there are exactly eleven types of Archimedean tilings on the plane (see Example \ref{exam:plane}). These types are $[3^6]$, $[4^4]$, $[6^3]$, $[3^4,6^1]$, $[3^3,4^2]$,  $[3^2,4^1,3^1,4^1]$, $[3^1,6^1,3^1,6^1]$, $[3^1,4^1,6^1,4^1]$, $[3^1,12^2]$,  $[4^1,6^1,12^1]$, $[4^1,8^2]$.
Clearly, Archimedean tilings and semi-equivelar maps on $\mathbb{R}^2$ are same. 
%But, there are semi-equivelar maps on $\mathbb{R}^2$ which are not (not isomorphic to) Archimedean tilings. In fact, there exists $[p^q]$ equivelar maps on $\mathbb{R}^2$ whenever $1/p+1/q < 1/2$ (e.g., \cite{CM1957}, \cite{FT1965}). 
We know from \cite{DU2005, DM2017, DM2018, GS1977}  that the Archimedean tilings $E_i$ $(1 \le i \le 11)$ (in Section \ref{sec:examples}) are unique and vertex-transitive. Thus, we have the following result. 

% \begin{proposition} \label{theo:plane}
% Let $E_1, \dots, E_{11}$ be the semi-equivelar maps on the plane given in Example $\ref{exam:plane}$.  If the type of $X$ is same as the type of $E_i$, for some $i\leq 11$, then $X\cong E_i$. In particular, $X$ is vertex-transitive.
% \end{proposition}

% As a consequence of Proposition \ref{theo:plane} we have 

\begin{proposition} \label{propo-1}
All semi-equivelar maps on the torus are the quotient of an Archimedean tiling on the plane by a discrete subgroup of the automorphism group of the tiling.
\end{proposition}

A map $M$ is said to be {\em vertex-transitive} if the automorphism group ${\rm Aut}(M)$ acts transitively on the set $V(M)$ of vertices of $M$. Clearly, vertex-transitive maps are semi-equivelar. In this context, Altshuler \cite{A1973} has shown construction and enumeration of maps of types $[3^{6}]$ and $[6^{3}]$ on the torus. Kurth \cite{kurth:1986} has given an enumeration of semi-equivelar maps of types $[3^{6}]$, $[4^{4}]$, $[6^{3}]$ on the torus. Negami \cite{negami:1983} has studied the uniqueness and faithfulness of embedding a class of toroidal graphs. 
Coxeter and Moser \cite[Chapter 8]{CM1957} have discussed the techniques to classify regular (rotary) maps on the torus. Then, Brehm and K\"{u}hnel \cite{brehm:2008} and Hubard et al. \cite{hubard:2012} have extended the techniques to classify symmetric toridal maps. In particular, Hubard et al. \cite[Lemma 6]{hubard:2012} have extended almost directly to Archimedean maps (see \cite{drach:2019}) and a classification of vertex-transitive semi-equivelar maps can be easily derived using
those techniques. Maity and Upadhyay \cite{MU2018} devised a way of enumerating all semi-equivelar maps of types $[3^6],$ $[4^4],$ $[6^3],$ $[3^3,4^2],$ $[3^2,4^1,3^1,4^1],$ $[3^4,6^1],$ $[4^1,6^1,12^1],$ $[3^1,6^1,3^1,6^1],$ $[3^1,4^1,6^1,4^1]$,  $[3^1,12^2]$ and $[4^1,8^2]$ on the torus. Datta and Upadhyay \cite{DU2005} have proved that maps of types $[3^6]$, $[6^3]$ and $[4^4]$ are vertex-transitive. Datta and Maity \cite{DM2017} have extended this result and have proved that the maps of types $[3^3, 4^2]$ are vertex-transitive on the torus and there are maps in the remaining seven types which are not vertex-transitive. In these seven cases, Pellicer  and  Weiss \cite{pel2011}, and \v{S}uch \cite{su:vt11} have shown the existence of a vertex-transitive map for each type on a given number of vertices.  We also know from \cite{DM2017}, \cite{DM2018} the following. 

\begin{proposition} \label{prop:no-of-orbits}
Let $X$ be a semi-equivelar map on the torus. Let the vertices of $X$ form $m$ ${\rm Aut}(X)$-orbits. 
%{\rm (a)} If the type of $X$ is  $[3^2,4^1,3^1,4^1]$ then $m\leq 2$.
%{\rm (b)} If the type of $X$ is  $[3^1,6^1,3^1,6^1]$ then $m\leq 3$.
{\rm (a)} If the type of $X$ is  $[3^2,4^1,3^1,4^1]$ or $[4^1,8^2]$ then $m\leq 2$.
{\rm (b)} If the type of $X$ is  $[3^1,6^1,3^1,6^1]$, $[3^4,6^1]$, $[3^1,4^1,6^1,4^1]$ or $[3^1,12^2]$ then $m\leq 3$. {\rm (c)} If the type of $X$ is  $[4^1,6^1, 12^1]$ then $m\leq 6$. In (a), (b) $\&$ (c), the bounds are sharp. 
\end{proposition}

%We begin with some definitions. 
Let $X_1 (= E/K_1), X_2 (= E/K_2)$ (by Prop. \ref{propo-1}) be two toroidal maps. $X_1, X_2$ are said to be {\em isomorphic} if there exists an isomorphism between $X_1, X_2$. $X_1, X_2$ are said to be {\em equal} if the orbits of $E$ under the action of $K_1, K_2$ are equal as sets. 
Let $K$ be a discrete, fixed point free, rank two subgroup of isometry group of $\mathbb{R}^2$ generated by two translations by vectors $w_1$ and $w_2$. Then the set $\mathbb{Z}w_1+\mathbb{Z}w_2$ is called {\em lattice} of $K$.

% \begin{definition}
% A toroidal map $X$ is called {\em $k$-orbital} if it has $k$ number of vertex orbits under the action of automorphism group Aut($X$) on the set of vertices of $X$.
% \end{definition}

% \begin{definition}
% A function %(number theoretic) 
% $f \colon \mathbb{N} \to \mathbb{N}$ is called {\em multiplicative} if for any two co-prime positive integers $m$ and $n$, $f(mn)=f(m)f(n)$.
% \end{definition}

% \begin{definition}
 
% \end{definition}

% \begin{definition}

% \end{definition}

We prove Theorems \ref{thm-main0}  - \ref{thm-main4} in Section \ref{sec:proofs-1}. Here, we give brief ideas of the proofs of the Theorems \ref{thm-main1}, \ref{thm-main3}  with examples. The others are similar to these.  

In Theorem \ref{thm-main1}, we discussed the existence of orbital covering maps for a given map. The idea is as follows. Let $X$ be a toroidal map. Then, by Prop. \ref{propo-1}, $X = E/K$ for some discrete subgroup $K$ of Aut($E$). Thus, let $E/K_1$ and $E/K_2$ be two maps on the torus. We know from \cite{GS2016} that $E/{K}_1$ and $E/{K}_2$ are isomorphic if ${K}_1$ and ${K_2}$ are conjugate in Aut($E$). That is, $E/{K}_1$ and $E/{K}_2$ are isomorphic if and only if there exists a symmetry of Aut($E$) sending lattice of ${K}_1$ to lattice of ${K}_2$. A symmetry $\gamma \in $ Aut($E$) induces an automorphism of the map $E/{K}$ if $\gamma$ normalizes ${K}$, that is, if the lattice of ${K}$ remains invariant under $\gamma$. Geometrically it is equivalent to mapping fundamental domains of ${K}$ to  fundamental domains of ${K}$. %It may happen that different symmetries of $E$ induces same automorphism on the torus. 
Therefore, Aut($E/{K}$) is the group induces by the normalizer of ${K}$ in Aut(${K}$), that is, 
\begin{equation}\label{nor}
     {\rm Aut}(E/K) = \frac{{\rm Nor}_{Aut(E)}(K)}{K}, \cite{hubard:2012}.
\end{equation}
Clearly, for every subgroup $L$ of $K$, the group $K/L$ acts on $E/L$. Hence, we get covering $E/L \longrightarrow \frac{E/L}{K/L}$. We construct some suitable subgroups $L$ (translation group) and $G$  with $L \trianglelefteq G \le Aut(E)$ by solving some diophantine equations such that $E/L$ has desired number of orbits by the action of $G/L$ and we show that the number of orbits remain unchanged by the action of Aut$(E/L)$. Thus, we show the existence of lesser orbital cover. The details of these proofs are given in the proof of Theorem \ref{thm-main1}.  

In Theorem \ref{thm-main3}, we discussed the classification of covering maps for a given map. Here we explain the idea for maps of type $[3^1,4^1,6^1,4^1]$. Let $T$ be the tiling of $\mathbb{R}^2$ of type $[3^1,4^1,6^1,4^1]$ (see Fig. \ref{fig:1}). In the proof, first we represented toroidal maps with some special types of matrices. Since, $X=E/K$ is torus, so, $K$ is a rank two discrete subgroup of translation group $\mathcal{H}$. Let $A$ and $B$ be points as indicated in Fig. \ref{fig:1}. Then translations by the vectors $\overrightarrow{OA}$ and $\overrightarrow{OB}$ are symmetry of the tiling $T$. 
Let us denote the translations by vectors $\overrightarrow{OA}$ and $\overrightarrow{OB}$ by $A$ and $B$ respectively. Then the translation group, $\mathcal{H}$, is generated by $A$ and $B$. Since $K$ has rank two there exists two linearly independent vectors $w_1$ and $w_2$ in $\mathbb{R}^2$ such that $K=\langle w_1,w_2 \rangle$. The set $\{A,B\}$ in tiling $T$ (in Fig. \ref{fig:1}) forms a basis of the vector space $\mathbb{R}^2$ over $\mathbb{R}$. Therefore there exists integers $p_1,q_1,p_2,q_2$ such that $w_1 = p_1A+q_1B$ and $w_2 = p_2A+q_2B$. Then, we represent $X$ by $\begin{bmatrix}
p_1&p_2\\q_1&q_2
\end{bmatrix}(=M$, say). 
%It may happen that distinct matrices represent equal maps. For that in Lemma \ref{equl}
Then, we have considered the hermite normal form of $M$, it will be of the form $\begin{bmatrix}
a&0\\b&d
\end{bmatrix}$ with $a,d>0$ and $0\le b<d$. Using this, we have shown that presence of a particular symmetry in Aut($X$) is equivalent to, $b$ is a solution of a polynomial equation modulo $d/a$ where $d\mid a$. Then by theory of quadratic residues we get some functions, denoted by $f_i(n)$ in Theorem \ref{thm-main3}, which gives number of solutions of these polynomials. $f_i(n)$ is depends on powers of different prime numbers in the factorization of $n$. $f_i(n)$ is number of $n$-sheeted covers having that particular symmetry in its automorphism group. However in this counting some isomorphic maps are counted more than once. To avoid that we have used the concept of stabilizer of group actions and its index. The details are given  in Section \ref{sec:proofs-1} and in the proof of Theorem \ref{thm-main3}. 

\begin{figure}[ht!]\label{T}
\tiny
\tikzstyle{ver}=[]
\tikzstyle{vert}=[circle, draw, fill=black!100, inner sep=0pt, minimum width=4pt]
\tikzstyle{vertex}=[circle, draw, fill=black!00, inner sep=0pt, minimum width=4pt]
\tikzstyle{edge} = [draw,thick,-]
\centering

%%%%%%%%%%%%%%%
\begin{tikzpicture}[scale=0.4] 

\draw ({sqrt(3)}, 1) -- (0, 2) -- ({-sqrt(3)}, 1) -- ({-sqrt(3)}, -1) -- (0, -2) -- ({sqrt(3)}, -1) -- ({sqrt(3)}, 1);
\draw ({6+sqrt(3)}, 1) -- (6+0, 2) -- ({6-sqrt(3)}, 1) -- ({6-sqrt(3)}, -1) -- (6+0, -2) -- ({6+sqrt(3)}, -1) -- ({6+sqrt(3)}, 1);
\draw ({12+sqrt(3)}, 1) -- (12+0, 2) -- ({12-sqrt(3)}, 1) -- ({12-sqrt(3)}, -1) -- (12+0, -2) -- ({12+sqrt(3)}, -1) -- ({12+sqrt(3)}, 1);

\draw ({-4.8+sqrt(3)}, 1) -- ({-sqrt(3)}, 1);
\draw ({-4.8+sqrt(3)}, -1) -- ({-sqrt(3)}, -1);
\draw ({sqrt(3)}, 1) -- ({6-sqrt(3)}, 1);
\draw ({sqrt(3)}, -1) -- ({6-sqrt(3)}, -1);
\draw ({6+sqrt(3)}, 1) -- ({12-sqrt(3)}, 1);
\draw ({6+sqrt(3)}, -1) -- ({12-sqrt(3)}, -1);
\draw ({12+sqrt(3)}, 1) -- ({16.8-sqrt(3)}, 1);
\draw ({12+sqrt(3)}, -1) -- ({16.8-sqrt(3)}, -1);

\draw ({-sqrt(3)}, 1) -- (-3+0, 2.6);
\draw (0, 2) -- (-1.25, 3.6);
\draw (0, 2) -- (1.3, 3.6);
\draw ({sqrt(3)}, 1) -- (3.05, 2.6);

\draw ({6-sqrt(3)}, 1) -- (3+0, 2.6);
\draw (6, 2) -- (4.75, 3.6);
\draw (6, 2) -- (7.3, 3.6);
\draw ({6+sqrt(3)}, 1) -- (9.05, 2.6);

\draw ({12-sqrt(3)}, 1) -- (9, 2.6);
\draw (12, 2) -- (10.75, 3.6);
\draw (12, 2) -- (13.3, 3.6);
\draw ({12+sqrt(3)}, 1) -- (15.05, 2.6);

%%%%

\draw [xshift = 86, yshift = 130] (-6+0, -2) -- ({-6+sqrt(3)}, -1) -- ({-6+sqrt(3)}, 1) -- ({-6+sqrt(3)}, 1) -- (-6+0, 2);
\draw [xshift = 86, yshift = 130] ({sqrt(3)}, 1) -- (0, 2) -- ({-sqrt(3)}, 1) -- ({-sqrt(3)}, -1) -- (0, -2) -- ({sqrt(3)}, -1) -- ({sqrt(3)}, 1);
\draw [xshift = 86, yshift = 130] ({6+sqrt(3)}, 1) -- (6+0, 2) -- ({6-sqrt(3)}, 1) -- ({6-sqrt(3)}, -1) -- (6+0, -2) -- ({6+sqrt(3)}, -1) -- ({6+sqrt(3)}, 1);
\draw [xshift = 86, yshift = 130] (12+0, 2) -- ({12-sqrt(3)}, 1) -- ({12-sqrt(3)}, -1) -- (12+0, -2) ;

\draw [xshift = 86, yshift = 130] ({-6+sqrt(3)}, 1) -- ({-sqrt(3)}, 1);
\draw [xshift = 86, yshift = 130] ({-6+sqrt(3)}, -1) -- ({-sqrt(3)}, -1);
\draw [xshift = 86, yshift = 130] ({sqrt(3)}, 1) -- ({6-sqrt(3)}, 1);
\draw [xshift = 86, yshift = 130] ({sqrt(3)}, -1) -- ({6-sqrt(3)}, -1);
\draw [xshift = 86, yshift = 130] ({6+sqrt(3)}, 1) -- ({12-sqrt(3)}, 1);
\draw [xshift = 86, yshift = 130] ({6+sqrt(3)}, -1) -- ({12-sqrt(3)}, -1);

%%%\tiny

\draw [yshift = 260] ({sqrt(3)}, 1) -- (0, 2) -- ({-sqrt(3)}, 1) -- ({-sqrt(3)}, -1) -- (0, -2) -- ({sqrt(3)}, -1) -- ({sqrt(3)}, 1);
\draw [yshift = 260] ({6+sqrt(3)}, 1) -- (6+0, 2) -- ({6-sqrt(3)}, 1) -- ({6-sqrt(3)}, -1) -- (6+0, -2) -- ({6+sqrt(3)}, -1) -- ({6+sqrt(3)}, 1);
\draw [yshift = 260]({12+sqrt(3)}, 1) -- (12+0, 2) -- ({12-sqrt(3)}, 1) -- ({12-sqrt(3)}, -1) -- (12+0, -2) -- ({12+sqrt(3)}, -1) -- ({12+sqrt(3)}, 1);

\draw [yshift = 260] ({-4.8+sqrt(3)}, 1) -- ({-sqrt(3)}, 1);
\draw [yshift = 260] ({-4.8+sqrt(3)}, -1) -- ({-sqrt(3)}, -1);
\draw [yshift = 260] ({sqrt(3)}, 1) -- ({6-sqrt(3)}, 1);
\draw [yshift = 260] ({sqrt(3)}, -1) -- ({6-sqrt(3)}, -1);
\draw [yshift = 260] ({6+sqrt(3)}, 1) -- ({12-sqrt(3)}, 1);
\draw [yshift = 260] ({6+sqrt(3)}, -1) -- ({12-sqrt(3)}, -1);
\draw [yshift = 260] ({12+sqrt(3)}, 1) -- ({16.8-sqrt(3)}, 1);
\draw [yshift = 260] ({12+sqrt(3)}, -1) -- ({16.8-sqrt(3)}, -1);

%%%%
\draw [xshift = 86, yshift = 130] ({-sqrt(3)}, 1) -- (-3+0, 2.6);
\draw [xshift = 86, yshift = 130](0, 2) -- (-1.25, 3.6);
\draw [xshift = 86, yshift = 130](0, 2) -- (1.3, 3.6);
\draw [xshift = 86, yshift = 130]({sqrt(3)}, 1) -- (3.05, 2.6);

\draw [xshift = 86, yshift = 130]({6-sqrt(3)}, 1) -- (3+0, 2.6);
\draw [xshift = 86, yshift = 130](6, 2) -- (4.75, 3.6);
\draw [xshift = 86, yshift = 130](6, 2) -- (7.3, 3.6);
\draw [xshift = 86, yshift = 130]({6+sqrt(3)}, 1) -- (9.05, 2.6);

\draw [xshift = 86, yshift = 130]({12-sqrt(3)}, 1) -- (9, 2.6);
\draw [xshift = 86, yshift = 130](12, 2) -- (10.75, 3.6);
\draw [yshift = 130](-3, 2) -- (-1.7, 3.6);
\draw [yshift = 130]({-3+sqrt(3)}, 1) -- (.05, 2.6);

%%%%

\draw [xshift = 86, yshift = -130] ({-.8-sqrt(3)}, 2) -- (-3+0, 2.6);
\draw [xshift = 86, yshift = -130](-.8, 3) -- (-1.3, 3.6);
\draw [yshift = -130](-2.2, 3) -- (-1.7, 3.6);
\draw [yshift = -130]({-2.2+sqrt(3)}, 2) -- (.05, 2.6);

\draw [xshift = 256, yshift = -130] ({-.8-sqrt(3)}, 2) -- (-3+0, 2.6);
\draw [xshift = 256, yshift = -130](-.8, 3) -- (-1.3, 3.6);
\draw [xshift = 170,yshift = -130](-2.2, 3) -- (-1.7, 3.6);
\draw [xshift = 170,yshift = -130]({-2.2+sqrt(3)}, 2) -- (.05, 2.6);

\draw [xshift = 426, yshift = -130] ({-.8-sqrt(3)}, 2) -- (-3+0, 2.6);
\draw [xshift = 426, yshift = -130](-.8, 3) -- (-1.3, 3.6);
\draw [xshift = 340,yshift = -130](-2.2, 3) -- (-1.7, 3.6);
\draw [xshift = 340,yshift = -130]({-2.2+sqrt(3)}, 2) -- (.05, 2.6);

%%%%
\draw [yshift = 260] ({-sqrt(3)}, 1) -- (-2.2+0, 1.6);
\draw [yshift = 260] (0, 2) -- (-0.5, 2.6);
\draw [yshift = 260] (0, 2) -- (.5, 2.6);
\draw [yshift = 260] ({sqrt(3)}, 1) -- (2.2, 1.6);

\draw [xshift = 170, yshift = 260] ({-sqrt(3)}, 1) -- (-2.2+0, 1.6);
\draw [xshift = 170, yshift = 260] (0, 2) -- (-0.5, 2.6);
\draw [xshift = 170, yshift = 260] (0, 2) -- (.5, 2.6);
\draw [xshift = 170, yshift = 260] ({sqrt(3)}, 1) -- (2.2, 1.6);

\draw [xshift = 340, yshift = 260] ({-sqrt(3)}, 1) -- (-2.2+0, 1.6);
\draw [xshift = 340, yshift = 260] (0, 2) -- (-0.5, 2.6);
\draw [xshift = 340, yshift = 260] (0, 2) -- (.5, 2.6);
\draw [xshift = 340, yshift = 260] ({sqrt(3)}, 1) -- (2.2, 1.6);

\draw [thick, dotted] (3,4.5) -- (9,4.5);
\draw [thick, dotted] (3,4.5) -- (6,9);
\draw [thick, dotted] (3,4.5) -- (0,9);

%\draw [very thick, dots=5 per 1cm] (3,4.5) -- (9,4.5);

%\draw [dashed] (3,4.5) -- (9,4.5);
%\draw [dashed] (3,4.5) -- (6,9);
%\draw [dashed] (3,4.5) -- (0,9);

\node[ver] () at (3,4.5){$\bullet$};
\put(25,50.8){\mbox{$O$}}
\node[ver] () at (9,4.5){$\bullet$};
\put(108,50.5){\mbox{$A$}}
\node[ver] () at (6,9){$\bullet$};
\put(75,105){\mbox{$B$}}
\node[ver] () at (0,9){$\bullet$};

%\node[ver] () at (5.2, -4){\normalsize };

\end{tikzpicture}
\caption{\bf $T([3^1, 4^1, 6^1, 4^1])$}\label{fig:1}
\end{figure}
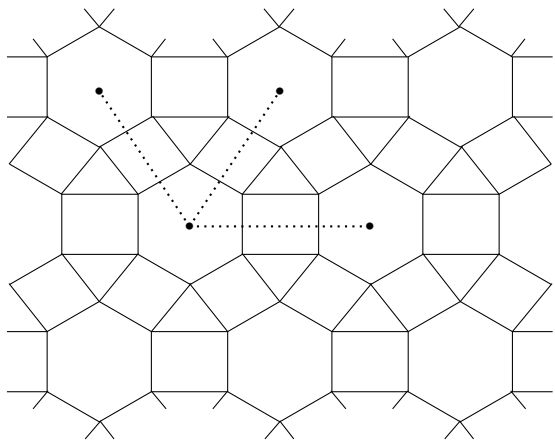

We illustrate Theorems \ref{thm-main1} and \ref{thm-main2} with a map of type $[3^1,4^1,6^1,4^1]$ on the torus.
$M_3$ (in \cite[Figure 2]{DM2018}) is a toroidal map with $3$ vertex orbits. The matrix representation of $M_3$ is $\begin{bmatrix}
2&1\\0&1
\end{bmatrix} (=A)$ and  $\begin{bmatrix}
2&0\\0&1
\end{bmatrix} (=B)$ is its Hermite normal form. 
By the ideas given in the proof of Theorem \ref{thm-main1}, we get a map which is represented by the matrix $C=\begin{bmatrix}
0&4\\2&2
\end{bmatrix}$ from $B$ where the corresponding toroidal map of $C$ is cover of the map corresponding of $B$. The hermite normal form of $C$ is $\begin{bmatrix}
2&0\\2&4
\end{bmatrix}$. From \cite[Lemma 3.1,3.3,3.4 and 3.9]{DKM2022} one can observe that the automorphism group of the map represented by $C$ generated by point reflection, translations and only one line reflection. The $60^{\circ}$ rotation and other two line reflections are not symmetry of the map. From this we can conclude that the map represented by $C$ is 2-orbital.
This $C$ is a $4$-sheeted cover of $M_3$ as det($C$)/det($A$)$=4$. Now, $f_1(4) = 1$ and $\sigma(4) = 1+2+4=7$ (see in the proof of Theorem \ref{thm-main2}). Therefore, by ideas given in the proof of Theorem \ref{thm-main2}, total number of $4$-sheeted covers up to isomorphism is $(7+2\cdot1)/3=3$. 
%So there are two more $4$-sheeted covers of $M_3$ other than $C$. 
%Similarly, as above, $C_1, C_2$ are the other two $4$-sheeted covers of $M_3$. 

\section{Examples: Archimedean tilings} \label{sec:examples}

We first present eleven Archimedean tilings on the plane. We need these examples for the proofs of our results in Section \ref{sec:proofs-1}.

\begin{example} \label{exam:plane}
{\rm Eleven Archimedean tilings on the plane are given in Fig. \ref{fig:Archi}. These are all the Archimedean tilings on the plane $\mathbb{R}^2$ (cf. \cite{GS1977}). All of these are vertex-transitive maps. %We need these examples for the proofs of our results in Section \ref{sec:proofs-1}. %Fig. \ref{fig:Archi} is also used in the proof of Claim 1 of Lemma \ref{lemma:U0}.
}
\end{example}

\bigskip

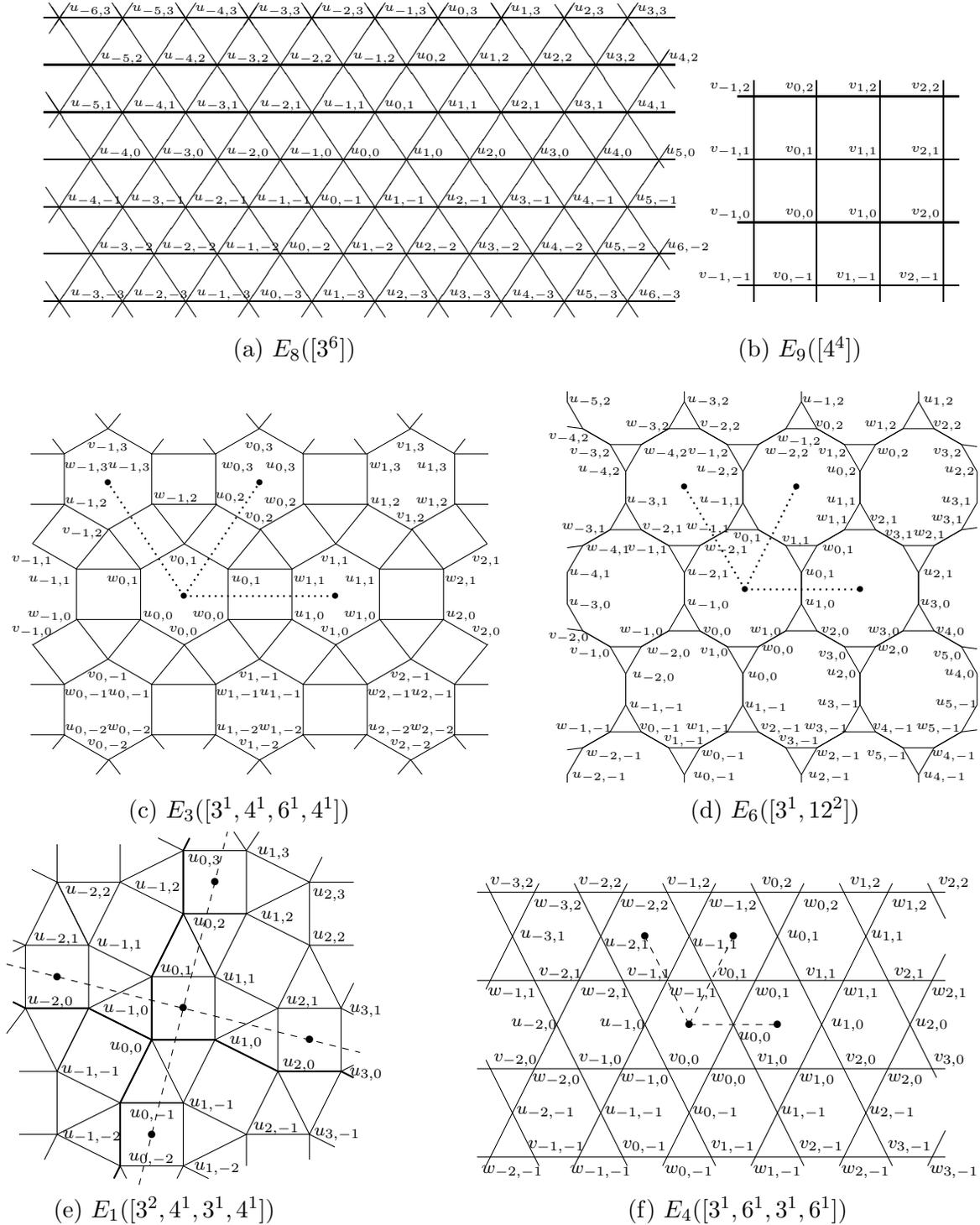
\setlength{\unitlength}{2.5mm}
%

%%%%%%%%%%%%%%%%%%%%%%%%%%%%%%%%%%%%%%%%%%%%%%%%%%%%%%%%%

\begin{figure}[ht!]
\tiny
\tikzstyle{ver}=[]
\tikzstyle{vert}=[circle, draw, fill=black!100, inner sep=0pt, minimum width=4pt]
\tikzstyle{vertex}=[circle, draw, fill=black!00, inner sep=0pt, minimum width=4pt]
\tikzstyle{edge} = [draw,thick,-]
\centering

%%%%%%%%%%%%%%%
%

%\end{figure}

%\smallskip

%%%%%%%%%%%%%%%%%%%%%%%%%E3E4%%%%%%%%%%%%%%%%%%%%%%%%%%%%%%%

%%%%%%%%%%%%%%%%%%%%%%%%%%%%%%%%%%%%%%%%%%%%%%%%%%%%%%

%%%%%%%%%%%%%%%%%%%%%%%E5&E6%%%%%%%%%%%%%%%%%%%%%%%%%%

%\begin{figure}[ht!]
\tiny
\tikzstyle{ver}=[]
\tikzstyle{vert}=[circle, draw, fill=black!100, inner sep=0pt, minimum width=4pt]
\tikzstyle{vertex}=[circle, draw, fill=black!00, inner sep=0pt, minimum width=4pt]
\tikzstyle{edge} = [draw,thick,-]
\centering

%
%%%%%%%%%%%%%%%%%%%%%%%%%%%%%%%%%%%%%%%%%%%%%%%%%%%%%%%%%

%%%%%%%%%%%%%%%%%%%%%%%E10&E11%%%%%%%%%%%%%%%%%%%%%

\begin{figure}[ht!]
\tiny
\tikzstyle{ver}=[]
\tikzstyle{vert}=[circle, draw, fill=black!100, inner sep=0pt, minimum width=4pt]
\tikzstyle{vertex}=[circle, draw, fill=black!00, inner sep=0pt, minimum width=4pt]
\tikzstyle{edge} = [draw,thick,-]
\centering

%

\end{figure}

\begin{figure}[ht!]
\tiny
\tikzstyle{ver}=[]
\tikzstyle{vert}=[circle, draw, fill=black!100, inner sep=0pt, minimum width=4pt]
\tikzstyle{vertex}=[circle, draw, fill=black!00, inner sep=0pt, minimum width=4pt]
\tikzstyle{edge} = [draw,thick,-]
%\centering

%
\caption{{\bf Archimedean tilings on the plane $\mathbb{R}^2$}}\label{fig:Archi}
\end{figure}

\section{Semi-equivelar maps and classification of their vertex covers}\label{sec:proofs-1}

For $1 \le i \le 11$, let $E_i$ be the Archimedean tiling given in Figure \ref{fig:Archi}. Let $V_i = V(E_i)$ be the vertex set of $E_i$. Let $\mathcal{H}_i$ be the group of all the translations of $E_i$. So, $\mathcal{H}_i \le Aut(E_i)$.
Let $X$ be a semi-equivelar toroidal map of vertex type $[3^6]$ or $[6^3]$ or $[4^4]$ or $[3^3, 4^2]$. Then from  \cite{DM2017} and \cite{DU2005}, $X$ is vertex-transitive.
Let $X$ be a semi-equivelar toroidal map of vertex type $[3^2, 4^1, 3^1, 4^1]$ or $[4^1, 8^2]$. Let vertices of $X$ forms $m$ Aut($X$) orbits. Then, $m \le 2$ by Prop. \ref{prop:no-of-orbits} and there exists a vertex-transitive cover of $X$ by Prop. \ref{datta2020}. 
Now, for the types $[3^1,4^1,6^1,4^1]$, $[3^1,6^1,3^1,6^1]$, $[3^1, 12^2]$, $[4^1,6^1,12^1]$, we prove Lemmas \ref{36} - \ref{4612}. We use these lemmas to proof Theorem \ref{thm-main1}.

%\subsection{Maps of type }\label{48}
%Let X be a semi-equivelar toroidal map of vertex type $[4^1, 8^2]$. Let vertices of X forms m Aut(X) orbits. From \cite{DM2018} we know that $m \le 2$ and there exists 2-orbital map. From \cite{BD2020} we know that there exists vertex-transitive covers for this type of maps. Hence, here also we have nothing to prove.

\begin{lemma}\label{36}
Let $X$ be a $3$-orbital semi-equivelar toroidal map of vertex type $[3^1,4^1,6^1, 4^1]$. Then there exists a $2$-orbital covering $Y\longrightarrow X.$
\end{lemma}
\begin{proof}
Let $X$ be a $3$-orbital semi-equivelar toroidal map of vertex type $[3^1,4^1,6^1, 4^1]$. By Proposition \ref{propo-1} we can assume $X=E_3/\Kc_3$ where $E_3$ (in Fig. \ref{fig:Archi}) is a semi-equivelar map of same vertex type on $\mathbb{R}^2$. $E_3$ is normal cover of $X$ as it is the universal cover. So we can take $\Kc_3$ to be the group of \emph{deck transformations} (\cite[Chapter 1]{AH2002}) of 
\begin{tikzcd}
E_3 \arrow[r, "p_3"] & X.
\end{tikzcd}
Hence, $\Kc_3$ has no fixed point. So $\Kc_3$ consists of translations and glide reflections. Since, $X=E_3/\Kc_3$ is orientable, $\Kc_3$ does not contain any glide reflection. Thus $\Kc_3 \leq \Hc_3$. Let $p_3 : E_3 \to X$ be the canonical covering map. We take $(0,0)$ as the middle point of the line segment joining vertices $u_{0,0}$ and $u_{0,1}$ of $E_3$. Let $A_3,B_3,F_3$ be the vectors form $(0,0)$ to middle point of the line segments joining $u_{1,0}$ and $u_{1,1}; ~ u_{0,2}$ and $u_{0,3};~ u_{-1,2}$ and $u_{-1,3}$ respectively. Then $A_3+F_3=B_3$. Observe that translation by these vectors are smallest possible translational symmetries of $E_3$. Hence, $$\Hc_3= \langle \alpha:z\mapsto z+A_3, \beta:z\mapsto z+B_3 \rangle.$$
Let $\mathcal{K}_3 = \langle \gamma, \delta \rangle \le H_1$ where $\gamma = \alpha^a\circ \beta^b$ and $\delta = \alpha^c\circ \beta^d$ for some $a,b,c,d \in \ZZ$.
Let $ \rho_3 $ be the rotation of $E_3$ by $180$ degrees anticlockwise with respect to origin. Then $\rho_3 \in Aut(E_3)$. Consider the group,
$$\mathcal{G}_3:= \langle \alpha_3,\beta_3,\rho_3 \rangle \leq Aut(E_3).$$
Now, observe that $E_3$ forms $6$ $\Hc_3$-orbits. These are $\mathcal{O}$($u_{0,1}$), $\mathcal{O}$($v_{0,1}$), $\mathcal{O}$($w_{0,1}$), $\mathcal{O}$($u_{0,0}$), $\mathcal{O}$($v_{0,0}$) and $\mathcal{O}$($w_{0,0}$). Now, $\rho_3(u_{0,1})=u_{0,0}, \rho_3(v_{0,1})=u_{0,0}, \rho_3(w_{0,1})=w_{0,0}$. Hence, $E_3$ forms three $\mathcal{G}_3$ orbits.
As $\Kc_3$ contains only translations so $\Kc_3 \leq \mathcal{G}_3$ and also $\Kc_3 \trianglelefteq \mathcal{G}_3$ because $\rho_3\gamma \rho_3^{-1} = \gamma^{-1}$ and $\rho_3\delta \rho_3^{-1} = \delta^{-1}$.
Therefore, $\mathcal{G}_3/\Kc_3$ acts on $X=E_3/\Kc_3$ and $X$ has three $\mathcal{G}_3/\Kc_3$-orbits. $X$ also has three Aut($X$)$=Aut(E_3/\Kc_3)$-orbits and $\mathcal{G}_3/\Kc_3 \leq Aut(E_3/\Kc_3)$.  Hence, for all $\alpha \in Aut(X) \setminus \mathcal{G}_3/\Kc_3, ~~\alpha(\mathcal{O}) = \mathcal{O}$ for all $\mathcal{G}_3/\Kc_3$-orbit $\mathcal{O}$ of $X$. Let $\tau_3$ be the map on $E_3$ obtained by taking reflection about the line passing through the points $u_{0,0}$ and $u_{0,1}$.
Now, consider, $$ \mathcal{G}'_3 = \langle \alpha_3, \beta_3, \rho_3, \tau_3 \rangle .$$
Clearly, $V(E_3)$ forms two $\mathcal{G}'_3$-orbits. 
$\gamma : z \mapsto z+ C$ and $\delta: z \mapsto z+ D$ where $C = aA_3 + bB_3$ and $D = cA_3+dB_3$.
Now, any cover of $X$ will be of the form $E_3/\mathcal{L}_3$ for some subgroup $\mathcal{L}_3$ of $\mathcal{K}_3$ with rank($\mathcal{L}_3$)$ = 2.$ 
We need $E_3/\mathcal{L}_3$ to be two orbital, i.e. $V(E_3/\mathcal{L}_3)$ has 2 orbits under Aut($E_3/\mathcal{L}_3)={\rm Nor}_{Aut(E_3)}(\mathcal{L}_3)/\mathcal{L}_3$ (See (\ref{nor})). Thus, if we get some $\mathcal{L}_3$ such that $\mathcal{L}_3 \trianglelefteq \mathcal{G}_3'$  then $\mathcal{G}_3'/\mathcal{L}_3$ acts on $V(E_3/\mathcal{L}_3)$ and give 2 orbits. Furthermore if we have Nor$(\mathcal{L}_3)=\mathcal{G}_3'$ then $E_3/\mathcal{L}_3$ will become 2-orbital. To proceed in this direction we made our first two claims as follows.

\begin{claim}\label{clm1}
There exists $m_1,n_1,m_2,n_2 \in \ZZ$ such that $\mathcal{L}_3 := \langle \gamma^{m_1}\delta^{n_1}, \gamma^{m_2}\delta^{n_2} \rangle \trianglelefteq \mathcal{G}_3'.$
\end{claim}
\noindent\emph{Proof of Claim \ref{clm1} :}
Suppose $\mathcal{L}_3 = \langle \gamma^{m_1}\delta^{n_1}, \gamma^{m_2}\delta^{n_2} \rangle$. We show that there exists suitable $m_1,n_1,m_2,n_2$ such that $\mathcal{L}_3 \trianglelefteq \mathcal{G}_3'$. %It turns out that we can take $m_1 = m_2$.
To satisfy $\mathcal{L}_3 \trianglelefteq \mathcal{G}_3'$ it is enough to show that $\rho_1\gamma^{m_1}\delta^{n_1}\rho_1^{-1}, \rho_1\gamma^{m_2}\delta^{n_2}\rho_1^{-1} \in \mathcal{L}_3$. It is known that conjugation of a translation by rotation and reflection is also a translation by rotated and reflected vector respectively. Since $\gamma$ and $\delta$ are translation by vectors $C$ and $D$ respectively, so $\gamma^{m_1}\delta^{n_1}$ and $\gamma^{m_2}\delta^{n_2}$ are translation by vectors $m_1C+n_1D$ and $m_2C+n_2D$ respectively. Hence, $\rho_1\gamma^{m_1}\delta^{n_1}\rho_1^{-1}$ and $\rho_1\gamma^{m_2}\delta^{n_2}\rho_1^{-1}$ are translation by the vectors $C'$ and $D'$ respectively. Where $C' = \rho_1(m_1C+n_1D) = \rho_1(m_1(aA_3 + bB_3)+n_1(dB_3)) = m_1a\rho_1(A_3) + m_1b\rho(B_3) + n_1d\rho(B_3) = m_1aB_3 + m_1bA_3 + n_1dA_3 = (m_1b+n_1d)A_3+m_1aB_3$ and similarly $D' = \rho_1(m_2C+n_2D) = (m_2b+n_2d)A_3+m_2aB_3$.
Now, these translations belong to $\mathcal{L}_3$ if the vectors $C'$ and $D'$ belong to lattice of $\mathcal{L}_3 = \ZZ(m_1C+n_1D) + \ZZ(m_2C+n_2D)$.
Suppose they belongs. Then $\exists~ p_1, q_1, p_2, q_2 \in \ZZ$ such that 
$$C' = p_1(m_1C+n_1D) + q_1(m_2C+n_2D), ~ D' = p_2(m_1C+n_1D) + q_2(m_2C+n_2D).$$
Putting expressions of $C', D', C, D$ in above equations and using the fact that $\{A_3, B_3\}$ is a linearly independent set we have,
$$p_1m_1a+q_1m_2a=m_1b+n_1d,~~ p_1(m_1b+n_1d)+q_1(m_2b+n_2d)=m_1a;$$
$$p_2m_1a+q_2m_2a=m_2b+n_2d,~~ p_2(m_1b+n_1d)+q_2(m_2b+n_2d)=m_2a.$$
%Since, $\{A_3, B_3\}$ is a linearly independent set we have,
%$$pam_1 + qcm_2 = bm_1, ~ pbm_1 + qdm_2 = am_1, ~sbm_1 + tdm_2 = cm_2, ~ sam_1 + tcm_2 = dm_2.$$
Now, as rank($\mathcal{L}_3$)$=2$ so $m_1n_2-n_1m_2 \neq 0$. 
%Dividing the above system by $m_1m_2$ we get,
%$$\frac{pa}{m_2} + \frac{qc}{m_1} = \frac{b}{m_2}, ~\frac{pb}{m_2} + \frac{qd}{m_1} = \frac{a}{m_2}, ~\frac{sb}{m_2} + \frac{td}{m_1} = \frac{c}{m_1}, ~\frac{sa}{m_2} + \frac{tc}{m_1} = \frac{d}{m_1}. $$
Now, consider $p_1, q_1, p_2, q_2$ as variables. We can treat above system as a system of linear equations. We can write these systems in matrix form as follows.
$$\begin{bmatrix} m_1a & m_2a &0&0\\ 
m_1b+n_1d & m_2b+n_2d &0 &0 \\
0& 0& m_1a & m_2a \\
0&0& m_1b+n_1d & m_2b+n_2d\end{bmatrix} 
\begin{bmatrix}
p_1\\q_1\\p_2\\q_2
\end{bmatrix} =
\begin{bmatrix}
m_1b+n_1d\\m_1a\\m_2b+n_2d\\m_2a
\end{bmatrix}$$
Now, $C,D$ are linearly independent thus $ad \neq 0$. Hence, the coefficient matrix of the above system has non zero determinant. Therefore the system has an unique solution. After solving we get,
$$p_2 = \frac{(m_2b+n_2d)^2-m_2^2a^2}{ad(m_1n_2-n_1m_2)}
,~~q_2 = \frac{(m_2b+n_2d)(m_1b+n_1d)-a^2m_1m_2}{ad(m_1n_2-n_1m_2)},$$
$$p_1 = -q_2
,~~q_1 = \frac{-(m_1b+n_1d)^2+m_1^2a^2+2n_1^2d^2}{ad(m_1n_2-n_1m_2)}.$$ 
Now, if we take $(m_1, m_2,n_1,n_2)  = (0,d,a^2d,a-b)$ then $p_1, q_1, p_2, q_2 \in \ZZ$.
Let $\mathcal{L}_3 := \langle \gamma^{m_1}\delta^{n_1}, \gamma^{m_2}\delta^{n_2} \rangle$. Matrix of $\mathcal{L}_3$ becomes $\begin{bmatrix}
0&ad\\a^2d^2&ad
\end{bmatrix}$. Thus we have $\mathcal{L}_3 \trianglelefteq \mathcal{G}_3'$. Hence, our Claim \ref{clm1} proved. 

Now, $V(E_3/\mathcal{L}_3)$ forms 2 orbits under the action of $\mathcal{G}_3'/\mathcal{L}_3$. To see whether $E_3/\mathcal{L}_3$ is 2-orbital or not we make the following,

\begin{claim}\label{corr}
${\rm Nor}_{Aut(E_3)}(\mathcal{L}_3) = \mathcal{G}_3'$.
\end{claim}
\noindent\emph{Proof of Claim \ref{corr} :}
Clearly from \ref{clm1} we can see that $\mathcal{G}_3'\leq {\rm Nor}(\mathcal{L}_3)$. so it is enough to show that ${\rm Nor}(\mathcal{L}_3) \leq \mathcal{G}_3'$. Let $\lambda \in {\rm Nor}(\mathcal{L}_3)$. Since, ${\rm Nor}(\mathcal{L}_3)$ is a group and the translations $\alpha, \beta \in {\rm Nor}(\mathcal{L}_3)$ so without loss of generality we can assume $\lambda(0)=0,$ i.e. $\gamma \in {\rm Steb}_{Aut(E_3)}(O)$. Where ${\rm Steb}_{Aut(E_3)}(O)$ denotes the stabilizer of $O$ under the action of $Aut(E_3)$. Now, from symmetries of the tiling $E_3$ one can see that ${\rm Steb}_{Aut(E_3)}(O) = \{R_1,R_2,R_3,R_1',R_2',R_3',\rho_{60}, \rho_{120},\rho_{60}^{-1},\rho_{120}^{-1},$ $\chi,e\}$. Where $\rho_d$ denotes rotation of $E_3$ about origin by angle $d^{\circ}$ anticlockwise. $R_1$, $R_2$, $R_3$ denotes the reflections about the lines passing through origin along $\overrightarrow{OA_3}$, $\overrightarrow{OB_3}$, $\overrightarrow{OF_3}~(F_3:=A_3+B_3)$ respectively. $R_i'$ is the reflection of $E_3$ about the line perpendicular to $R_i$ for $i=1,2,3$. Therefore, it is enough to show $\lambda \notin \{R_2,R_3,R_1',R_3',\rho_{60}, \rho_{120},\rho_{60}^{-1},\rho_{120}^{-1},\}$.
Pictorially one can observe these lines in Figure \ref{refl}.
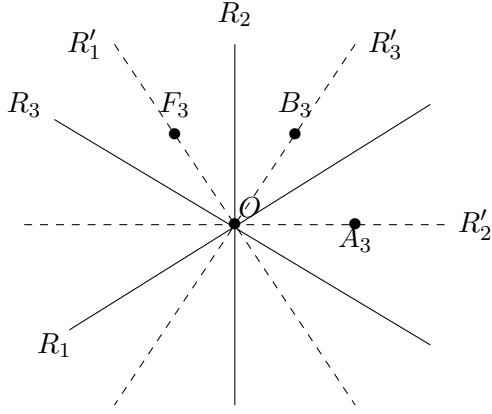
\begin{figure}
    \centering
    \begin{tikzpicture}[scale=0.8]
	    \node  (0) at (0, 0) {$\bullet$};
		\node  (1) at (3.5, 0) {};
		\node  (2) at (-3.5, 0) {};
		\node  (3) at (2, 3) {};
		\node  (4) at (-2, 3) {};
		\node  (5) at (-2, -3) {};
		\node  (6) at (2, -3) {};
		\node  (7) at (4, 0) {};
		\node  (8) at (4, 0) {$R_2'$};
		\node  (9) at (2.5, 3) {};
		\node  (10) at (2.5, 3) {$R_3'$};
		\node  (11) at (-2.5, 3) {$R_1'$};
		\node  (12) at (3.25, 2) {};
		\node  (13) at (0, 3) {};
		\node  (14) at (0, -3) {};
		\node  (15) at (-3, 1.75) {};
		\node  (16) at (-2.75, -1.75) {};
		\node  (17) at (3.25, -2) {};
		\node  (18) at (0, 3.5) {$R_2$};
		\node  (19) at (-3.5, 2) {$R_3$};
		\node  (20) at (-3, -2) {$R_1$};
		\node  (21) at (2, 0) {$\bullet$};
		\node  (22) at (1, 1.5) {$\bullet$};
		\node  (23) at (-1, 1.5) {$\bullet$};
		\node  (24) at (2, -0.25) {$A_3$};
		\node  (25) at (1, 2) {$B_3$};
		\node  (26) at (-1, 2) {$F_3$};
		\node  (27) at (0.25, 0.30) {$O$};
	
		\draw [style=dashed](2.center) to (1.center);
		\draw [style=dashed](3.center) to (5.center);
		\draw [style=dashed](4.center) to (6.center);
		\draw (13.center) to (14.center);
		\draw (12.center) to (16.center);
		\draw (15.center) to (17.center);
	\end{tikzpicture}
    \caption{Line reflections in $E_3$}
    \label{refl}
\end{figure}
Let matrix of $\mathcal{L}_3$ with respect to the basis $\{\alpha(0)=A_3,\beta(0)=B_3\}$ be  $\begin{bmatrix}
x&y\\z&w
\end{bmatrix}.$
Let us see when $R_2\in {\rm Nor}(\mathcal{L}_3)$. Here, $R_2(\alpha(0))=\beta(0)$ and $R_2(\beta(0))=\beta(0)-\alpha(0)$. Therefore the matrix of the transformation $R_2$ is $\begin{bmatrix}
0&-1\\-1&1
\end{bmatrix}$.
Now, $\begin{bmatrix}
0&-1\\-1&1
\end{bmatrix} \begin{bmatrix}
x&y\\z&w
\end{bmatrix} = \begin{bmatrix}
-z&-w\\z-x&w-y
\end{bmatrix}$.
Therefore $R_2 \in {\rm Nor}(\mathcal{L}_3) \iff \begin{bmatrix}
-z\\z-x
\end{bmatrix}$
and $\begin{bmatrix}
-w\\w-y
\end{bmatrix} \in \mathbb{Z}\begin{bmatrix}
x\\z
\end{bmatrix} + \mathbb{Z}\begin{bmatrix}
y\\w
\end{bmatrix}$.
Suppose they belong. Then $\exists$ $p,q,r,s \in \mathbb{Z}$ such that $-z=px+qy, z-x=pz+qw, -w=rx+sy, w-y=rz+sw$. Hence, we have the following system of linear equations,
$\begin{bmatrix}
x&y\\z&w
\end{bmatrix}\begin{bmatrix}
p\\q
\end{bmatrix} = \begin{bmatrix}
-z\\z-x
\end{bmatrix}, \begin{bmatrix}
x&y\\z&w
\end{bmatrix}\begin{bmatrix}
r\\s
\end{bmatrix} = \begin{bmatrix}
-w\\w-y
\end{bmatrix}$.
Since determinant of the coefficient matrix is non zero so these systems has unique solutions. After solving we get 
$p=\dfrac{xy-yz-zw}{xw-yz}, q=\dfrac{z^2+xz-x^2}{xw-yz}, r=\dfrac{y^2-w^2-yw}{xw-yz}, s=\dfrac{wz+xw-xy}{xw-yz}$.
Therefore, $R_2 \in {\rm Nor}(\mathcal{L}_3) \iff p,q,r,s \in \mathbb{Z}$.
Now, putting $(x,y,z,w) = (0,ad,a^2d^2,ad)$ we get $r=\frac{1}{ad}$. So $r \notin \mathbb{Z}$ unless $ad=1$. But if $ad=1$ then the matrix of our initial given map $X$ will be either $\begin{bmatrix}
1&0\\1&1
\end{bmatrix}$ or $\begin{bmatrix}
1&0\\0&1
\end{bmatrix}$.
So corresponding group by which we are quotienting $E_3$ will be $\langle \alpha+\beta , \beta \rangle$ and $\langle \alpha, \beta \rangle$ respectively. Since these two groups are equal corresponding maps are isomorphic. One can see easily that ${\rm Nor}(\langle \alpha, \beta \rangle) = {\rm Aut}(E_3)$. So it will be one orbital. This contradicts our assumption that $X$ is $3$-orbital. Hence, $ad \neq 1$. Therefore, $R_2 \notin {\rm Nor}(\mathcal{L}_3)$.
Similarly, for $R_3'$ we will get same conditions. Hence, $R_3' \notin {\rm Nor}(\mathcal{L}_3)$.
Proceeding in similar way $R_3$ and $R_1' \in {\rm Nor}(\mathcal{L}_3) \iff$ the followings are integers.
$\dfrac{wx+yx+yz}{xw-yz}, \dfrac{-x^2-2xz}{xw-yz}, \dfrac{y^2+2yw}{xw-yz}, \dfrac{-xy-yz-zw}{xw-yz}$. Putting $(x,y,z,w) = (0,ad,a^2d^2,ad)$ the third one we get $\frac{3}{ad}$. 
Similarly $\rho_{60}, \rho_{120} \in {\rm Nor}(\mathcal{L}_3) \iff $ the followings are integer.
$\dfrac{xy+yz+wz}{xw-yz}, \dfrac{-x^2-xz-z^2}{xw-yz}, \dfrac{y^2+yw+w^2}{xw-yz}, \dfrac{-xy-yz-zw}{xw-yz}$. Putting $(x,y,z,w) = (0,ad,a^2d^2,ad)$ the third one we get $\frac{3}{ad}$. 
The above entries are all become integer if $ad=1$ or $3$. But we observed, $ad \neq 1$. 
Now, if $ad=3$ then  only the maps corresponding to following matrices are possible.
$\begin{bmatrix}
3&0\\0&1
\end{bmatrix}, \begin{bmatrix}
1&0\\1&3
\end{bmatrix},\begin{bmatrix}
1&0\\2&3
\end{bmatrix},\begin{bmatrix}
1&0\\0&3
\end{bmatrix}$. One can easily check that first three matrices will not give $3$-orbital maps and last one is not polyhedral. Since the given map is polyhedral, $ad \neq 3$. Therefore, ${\rm Nor}(\mathcal{L}_3) \leq \mathcal{G}_3'$. This proves Claim \ref{corr}.

Now, action of $\mathcal{G}_3'$ on $V(E_3)$ gives $2$ orbits. From (\ref{nor}) we have $${\rm Aut}(E_3/\mathcal{L}_3) = {\rm Nor_{Aut(E_3)}}(\mathcal{L}_3)/\mathcal{L}_3 = \mathcal{G}_3'/\mathcal{L}_3.$$ Hence, the map $Y := E_3/\mathcal{L}_3$ is $2$-orbital.

Clearly $Y$ is a toroidal cover of $X$ and $V(Y)$ forms two orbits under that action on Aut$(Y)$. This proves our Lemma \ref{36}.
\end{proof}

                                        %            do diff

%\begin{claim} --about minimality of $Y$.
%\end{claim}
%--------------------------------------------------------------------------------------------------------------------++++---

\begin{lemma}\label{3636}
Let $X$ be a $3$-orbital semi-equivelar toroidal map of vertex type $[3^1, 6^1,3^1,6^1]$. Then there exists a $2$-orbital covering $Y\longrightarrow X.$
\end{lemma}
\begin{proof}
Let $X$ be a $3$-orbital semi-equivelar toroidal map of vertex type $[3^1, 6^1,3^1,6^1]$. By Proposition \ref{propo-1} we can assume $X=E_4/\Kc_4$ where $E_4$ (in Fig. \ref{fig:Archi}) is a semi-equivelar map of same vertex type on $\mathbb{R}^2$. $E_4$ is normal cover of $X$ as it is the universal cover so we can take $\Kc_4$ to be the group of deck transformations of 
\begin{tikzcd}
E_4 \arrow[r, "p_4"] & X
\end{tikzcd}.
Hence, $\Kc_4$ has no fixed point. So $\Kc_4$ consists of translations and glide reflections. Since, $X=E_4/\Kc_4$ is orientable, $\Kc_4$ does not contain any glide reflection. Thus $\Kc_4 \leq \Hc_4$. Let $p_4 : E_4 \to X$ be the canonical covering map. We take $(0,0)$ as the middle point of the line segment joining vertices $u_{-1,0}$ and $u_{0,0}$ of $E_4$. Let $A_4,B_4,F_4$ be the vectors form $(0,0)$ to middle point of the line segments joining $u_{0,0}$ and $u_{1,0}; ~ u_{-1,1}$ and $u_{0,1};~ u_{-2,1}$ and $u_{-1,1}$ respectively. Then $A_4+F_4=B_4$. Observe that translation by these vectors are smallest possible translational symmetries of $E_4$. Hence, 
$$\Hc_4= \langle \alpha_4:z\mapsto z+A_4, \beta_4:z\mapsto z+B_4 \rangle.$$
Let $ \rho_4 $ be the reflection of $E_4$ with respect to the line joining origin and $u_{-1,1}$. Then $\rho_4 \in Aut(E_4)$. Consider the group 
$$\mathcal{G}'_4:= \langle \alpha_4,\beta_4, \tau_4,\rho_4 \rangle \leq Aut(E_4).$$ Where $\tau_4$ is the map obtained by taking reflection about origin in $E_4$.
Let $\mathcal{G}_4 = \Hc_4 \rtimes \mathbb{Z}_2 $ where $\mathbb{Z}_2$ is generated by $\tau_4$.
Now, observe that $E_4$ forms $3$ $\mathcal{G}_4$-orbits. These are $\mathcal{O}$($u_{0,0}$), $\mathcal{O}$($v_{0,0}$), $\mathcal{O}$($w_{0,0}$). Now, $\rho_4(u_{0,0})=u_{-1,0}, \rho_4(v_{0,0})=w_{0,0}$. Hence, $E_4$ forms $2$ $\mathcal{G}'_4$ orbits.
As $\Kc_3$ contains only translations so  $\Kc_4 \trianglelefteq \mathcal{G}_4.$
Therefore, $\mathcal{G}_4/\Kc_4$ acts on $X=E_4/\Kc_4$ and $X$ has $3$ $\mathcal{G}_4/\Kc_4$-orbits. $X$ also has $3$ Aut($X$)$=Aut(E_4/\Kc_4)$-orbits and $\mathcal{G}_4/\Kc_4 \leq Aut(E_4/\Kc_4)$ Hence, for all $\alpha \in Aut(X) \setminus \mathcal{G}_4/\Kc_4, ~~\alpha(\mathcal{O}) = \mathcal{O}$ for all $\mathcal{G}_4/\Kc_4$-orbit $\mathcal{O}$ of $X$.\\
As $\Kc_4 \leq \Hc_4 = \langle \alpha_4,\beta_4 \rangle = \{ m\alpha_4+n\beta_4 \mid (m,n)\in \mathbb{Z}\times \mathbb{Z} \}$ so 
$$\Kc_4 = \langle \gamma:z\mapsto z+C_4,\delta:z\mapsto z+D_4 \rangle,$$
where $C_4 = aA_4+bB_4$ and $D_4=cA_4+dB_4$, for some $a,b,c,d \in \mathbb{Z}$.
Proceedings in the same way as in Lemma \ref{36} one can prove followings,

\begin{claim}\label{L4}
There exists $m_1,n_1,m_2,n_2 \in \ZZ$ such that $\mathcal{L}_4 := \langle \gamma^{m_1}\delta^{n_1}, \gamma^{m_2}\delta^{n_2} \rangle \trianglelefteq \mathcal{G}_4'.$
\end{claim}

\begin{claim}\label{clm3}
${\rm Nor}_{Aut(E_4)}(\mathcal{L}_4) = \mathcal{G}_4'$.
\end{claim}

% \begin{claim}
% Given $\alpha \in {\rm Aut}(E_4/\mathcal{L}_4)={\rm Nor_{Aut(E_4)}}(\mathcal{L}_4)/\mathcal{L}_4$ there exists $\widetilde{\alpha} \in {\rm Aut}(E_4/\mathcal{K}_4)$ such that $p\circ \alpha = \widetilde{\alpha} \circ p.$
% \end{claim}

% \begin{claim}
% If $\alpha \in Aut(Y) \setminus \frac{\mathcal{G}'_4}{\mathcal{L}_4}$ then $\alpha(\mathcal{O}) = \mathcal{O}$ for all $\frac{\mathcal{G}'_4}{\mathcal{L}_4}$-orbits $\mathcal{O}$ of $Y$.
% \end{claim}
%\begin{claim}\label{cvr2}
%$Y$ is a $2$-orbital cover of $X$.
%\end{claim}
Clearly $Y:E_4/\mathcal{L}_4$ is a toroidal cover of $X$ and one can show in similar way as in Lemma \ref{36} that $V(Y)$ forms two orbits under the action of Aut($Y$). Hence, our Lemma \ref{3636} is proved.
%\begin{claim} --about minimality of $Y$.
%\end{claim}
\end{proof}
%-----------------------------------------------------------------------------------------------------------++++++-------||-

% \begin{lemma}\label{3464}
% There does not exists a $2$-orbital toroidal map of type $[3^4, 6^1]$.
% %Let $X$ be a $3$-orbital semi-equivelar toroidal map of vertex type $[3^4, 6^1]$. Then there exists a $2$-orbital covering $Y\longrightarrow X.$
% \end{lemma}
% \begin{proof}
%\end{proof}
%------------------------------------------------------------------------------------------------------------++++++++---|||-

\begin{lemma}\label{312}
Let $X$ be a $3$-orbital semi-equivelar toroidal map of vertex type $[3^1, 12^2]$. Then there exists a $2$-orbital covering $Y\longrightarrow X.$
\end{lemma}
\begin{proof}

Let $X$ be a semi-equivelar toroidal map of vertex type $[3^1, 12^2]$. By Proposition \ref{propo-1} we can assume $X=E_6/\Kc_6$ where $E_6$ (in Fig. \ref{fig:Archi}) is a semi-equivelar map of same vertex type on $\mathbb{R}^2$. $E_6$ is normal cover of $X$ as it is the universal cover so we can take $\Kc_6$ to be the group of deck transformations of 
\begin{tikzcd}
E_6 \arrow[r, "p_6"] & X
\end{tikzcd}.
Hence, $\Kc_6$ has no fixed point. So $\Kc_6$ consist of translations and glide reflections. Since, $X=E_6/\Kc_6$ is orientable, $\Kc_6$ does not contain any glide reflection. Thus $\Kc_6 \leq \Hc_6$. Let $p_6 : E_6 \to X$ be the canonical covering map. We take $(0,0)$ as the middle point of the line segment joining vertices $v_{0,0}$ and $v_{1,1}$ of $E_6$. Let $A_6,B_6,F_6$ be the vectors form $(0,0)$ to middle point of the line segments joining $u_{1,0}$ and $u_{2,1}; ~ v_{0,1}$ and $v_{1,2};~ v_{-1,2}$ and $v_{-2,1}$ respectively. Then $A_6+F_6=B_6$. Observe that translation by these vectors are smallest possible translational symmetries of $E_6$. Hence, 
$$\Hc_6= \langle \alpha_6:z\mapsto z+A_6, \beta_6:z\mapsto z+B_6 \rangle.$$

Let $ \rho_6 $ be the map obtained by taking 180 degree rotation of $E_6$  with respect to the origin. Then $\rho_6 \in Aut(E_6)$. Consider the group 
$$\mathcal{G}_6:= \langle \alpha_6,\beta_6,\rho_6 \rangle \leq Aut(E_6).$$
Now, observe that $E_6$ forms $6$ $\Hc_6$-orbits. These are $\mathcal{O}$($u_{-2,1}$), $\mathcal{O}$($u_{-1,0}$), $\mathcal{O}$($v_{0,0}$), $\mathcal{O}$($v_{1,0}$), $\mathcal{O}$($w_{-2,1}$) and $\mathcal{O}$($v_{0,1}$). Now, $\rho_6(u_{-2,1})=u_{1,0}, \rho_6(v_{0,0})=v_{1,1}, \rho_6(v_{1,0})=v_{0,1}$. Hence, $E_6$ forms three $\mathcal{G}_6$ orbits.
As $\Kc_6$ contains only translations so $\Kc_6 \leq \mathcal{G}_6$ and also $\Kc_6 \trianglelefteq \mathcal{G}_6$ because of same reason as in Lemma \ref{36}.
Therefore, $\mathcal{G}_6/\Kc_6$ acts on $X=E_6/\Kc_6$ and $X$ has three $\mathcal{G}_6/\Kc_6$-orbits. $X$ also has $3$ Aut($X$)=Aut$(E_6/\Kc_6)$-orbits and $\mathcal{G}_6/\Kc_6 \leq Aut(E_6/\Kc_6)$ Hence, for all $\alpha \in Aut(X) \setminus \mathcal{G}_6/\Kc_6 ~~\alpha(\mathcal{O}) = \mathcal{O}$ for all $\mathcal{G}_6/\Kc_6$-orbit $\mathcal{O}$ of $X$.
Let $\tau_6$ be the map obtained by taking reflection of $E_6$ with respect to the line passing through origin and perpendicular to the vector $\overrightarrow{OA_6}$. Then $\tau_6 \in Aut(E_6)$. Now, consider,
$$\mathcal{G}'_6 = \langle \alpha_6, \beta_6, \tau_6, \rho_6 \rangle .$$
Clearly $E_6$ forms two $\mathcal{G}'_6$ orbits.
As $\Kc_6 \leq \Hc_6 = \langle \alpha_6,\beta_6 \rangle = \{ m\alpha_6+n\beta_6 \mid (m,n)\in \mathbb{Z}\times \mathbb{Z} \}$ so 
$$\Kc_6 = \langle \gamma:z\mapsto z+C_6,\delta:z\mapsto z+D_6 \rangle,$$
where $C_6 = aA_6+bB_6$ and $D_6=cA_6+dB_6$, for some $a,b,c,d \in \mathbb{Z}$.\\
Now, following the same procedure as in Lemma \ref{36} we can prove the following claims,

\begin{claim} 
There exists $m_1,n_1,m_2,n_2 \in \ZZ$ such that $\mathcal{L}_6 := \langle \gamma^{m_1}\delta^{n_1}, \gamma^{m_2}\delta^{n_2} \rangle \trianglelefteq \mathcal{G}_6'.$

\end{claim}

\begin{claim} \label{orb6}
${\rm Nor}_{Aut(E_6)}(\mathcal{L}_6) = \mathcal{G}_6'$.
\end{claim}

% \begin{claim}
% Given $\alpha \in {\rm Aut}(E_6/\mathcal{L}_6)={\rm Nor_{Aut(E_6)}}(\mathcal{L}_6)/\mathcal{L}_6$ there exists $\widetilde{\alpha} \in {\rm Aut}(E_6/\mathcal{K}_6)$ such that $p\circ \alpha = \widetilde{\alpha} \circ p.$
% \end{claim}

% \begin{claim}
% If $\alpha \in Aut(Y) \setminus \frac{\mathcal{G}'_6}{\mathcal{L}_6}$ then $\alpha(\mathcal{O}) = \mathcal{O}$ for all $\frac{\mathcal{G}'_6}{\mathcal{L}_6}$-orbits $\mathcal{O}$ of $Y$.
% \end{claim}

%\begin{claim} 
%$Y$ is a 2-orbital cover of $X$.
%\end{claim}
$V(Y)$ forms two $\mathcal{G}'_6/\mathcal{L}_6$-orbits and from Claim \ref{orb6} it follows that $V(Y)$ has two Aut($Y$) orbits. Thus $Y$ is a $2$-orbital cover of $X$. This proves our Lemma \ref{312}.
\end{proof}

%--------------------------------------------------------------------------------------------------------+++++++----||||---

\begin{lemma}\label{4612}
If $X$ is a $m$-orbital toroidal map of type $[4^1,6^1,12^1]$, then there exists a covering $\eta_{k} \colon Y_{k} \to X$ where $Y_{k}$ is $k$-orbital for each $k (\neq 4, 5) \le m$ except $(m, k) = (3,2)$.
\end{lemma}
\begin{proof}
Let $X$ be a semi-equivelar toroidal map of vertex type $[4^1, 6^1, 12^1]$. By Proposition \ref{propo-1} we can assume $X=E_7/\Kc_7$ where $E_7$ (in Fig. \ref{fig:Archi}) is a semi-equivelar map of same vertex type on $\mathbb{R}^2$. $E_7$ is normal cover of $X$ as it is the universal cover so we can take $\Kc_7$ to be the group of deck transformations of 
\begin{tikzcd}
E_7 \arrow[r, "p_7"] & X.
\end{tikzcd}
Hence, $\Kc_7$ has no fixed point. So $\Kc_7$ consists of translations and glide reflections. Since, $X=E_7/\Kc_7$ is orientable, $\Kc_7$ does not contain any glide reflection. Thus $\Kc_7 \leq \Hc_7$. Let $p_7 : E_7 \to X$ be the canonical covering map. We take $(0,0)$ as the middle point of the line segment joining vertices $u_{0,0}$ and $u_{0,1}$ of $E_7$. Let $A_7,B_7,F_7$ be the vectors form $(0,0)$ to middle point of the line segments joining $u_{1,0}$ and $u_{1,1}; ~ u_{0,2}$ and $u_{0,3};~ u_{-1,2}$ and $u_{-1,3}$ respectively. Then $A_7+F_7=B_7$. Observe that translation by these vectors are smallest possible translational symmetries of $E_7$. Hence, 
$$\Hc_7= \langle \alpha_7:z\mapsto z+A_7, \beta_7:z\mapsto z+B_7 \rangle.$$
As $\Kc_7 \leq \Hc_7 = \langle \alpha_7,\beta_7 \rangle = \{ m\alpha_7+n\beta_7 \mid (m,n)\in \mathbb{Z}\times \mathbb{Z} \}$ so 
$$\Kc_7 = \langle \gamma:z\mapsto z+C_7,\delta:z\mapsto z+D_7 \rangle,$$
where $C_7 = aA_7+bB_7$ and $D_7=cA_7+dB_7$, for some $a,b,c,d \in \mathbb{Z}$.
Suppose vertices of $X$ forms $m$ Aut($X$) orbits. From \cite{DM2018} we get $m\le 6$. Here the following questions comes.
Given $m$-orbital map does there exists a $k$-orbital cover?
for $(m,k)=(6,2),(6,3),(6,4),(6,5),(5,4),(5,$ $3) (5,2),(4,3),(4,2),(3,2).$ 
Let's answer this questions case by case.

\texttt{Case 1.}
Suppose $X$ is six orbital map.
Let $\tau_7$  be the map obtained by taking 180 degree rotation of $E_7$ about origin. Consider the group,
$$\mathcal{G}_7 = \langle \alpha_7,\beta_7,\tau_7 \rangle.$$
Now, vertices of $E_7$ forms $12$ $\Hc_7$ orbits. They are represented by vertices of any $12$-gon. Therefore, vertices of $E_7$ forms, $6$ $\mathcal{G}_7$ orbits. As $\Kc_7$ contains only translations so $\Kc_7 \leq \mathcal{G}_7$. $\Kc_7 \trianglelefteq \mathcal{G}_7.$
Therefore, $\mathcal{G}_7/\Kc_7$ acts on $X=E_7/\Kc_7$ and $X$ has six $\mathcal{G}_7/\Kc_7$-orbits. $X$ also has six Aut($X$)=Aut$(E_7/\Kc_7)$-orbits and $\mathcal{G}_7/\Kc_7 \leq Aut(E_7/\Kc_7)$ Hence, for all $\alpha \in Aut(X) \setminus \mathcal{G}_7/\Kc_7 ~~\alpha(\mathcal{O}) = \mathcal{O}$ for all $\mathcal{G}_7/\Kc_7$-orbit $\mathcal{O}$ of $X$.
Suppose $\gamma_7$ be the map obtained by taking reflection of $E_7$ about the line passing through origin and parallel to the vector $B_7$. Consider the group,
$$ \mathcal{G}'_7 := \langle \alpha_7,\beta_7,\tau_7, \gamma_7 \rangle.$$
Observe that $E_7$ forms two $ \mathcal{G}'_7$-orbits.
Now, proceeding in same way as in Lemma \ref{36} we can prove the following claims.
\begin{claim} 
There exists $m_1,n_1,m_2,n_2 \in \ZZ$ such that $\mathcal{L}_7 := \langle \gamma^{m_1}\delta^{n_1}, \gamma^{m_2}\delta^{n_2} \rangle \trianglelefteq \mathcal{G}_7'.$

\end{claim}

\begin{claim} \label{orb7}
${\rm Nor}_{Aut(E_7)}(\mathcal{L}_7) = \mathcal{G}_7'$.
\end{claim}

% \begin{claim}
% Given $\alpha \in {\rm Aut}(E_7/\mathcal{L}_7)={\rm Nor_{Aut(E_7)}}(\mathcal{L}_7)/\mathcal{L}_7$ there exists $\widetilde{\alpha} \in {\rm Aut}(E_7/\mathcal{K}_7)$ such that $p\circ \alpha = \widetilde{\alpha} \circ p.$
% \end{claim}

% \begin{claim}
% If $\alpha \in Aut(Y) \setminus \frac{\mathcal{G}'_7}{\mathcal{L}_7}$ then $\alpha(\mathcal{O}) = \mathcal{O}$ for all $\frac{\mathcal{G}'_7}{\mathcal{L}_7}$-orbits $\mathcal{O}$ of $Y$.
% \end{claim}
%\begin{claim} 
%$Y$ is a $2$-orbital cover of $X$.
%\end{claim}
$V(Y)$ forms $2$ $\mathcal{G}'_7/\mathcal{L}_7$-orbits and from Claim \ref{orb7} it follows that $V(Y)$ has $2$ Aut($Y$) orbits. Thus $Y$ is a $2$-orbital cover of $X$. 

%\begin{claim} --about minimality of $Y$.
%\end{claim}

{\texttt{Case 2.}}
Suppose $X$ is $6$ orbital map. Define $\mathcal{G}_7$ as in Case 1. Let $\omega_7$ be the map obtained by taking reflection about the line parpendicular to $OA_7$ and passing through origin. Here consider,
$$ \mathcal{G}'_7 := \langle \alpha_7,\beta_7,\tau_7, \omega_7 \rangle.$$
Observe that $E_7$ forms $3$ $ \mathcal{G}'_7$-orbits. The result follows in this case by exactly same way as in Case 1. Here we get $Y$ be a $3$-orbital cover of $X$. To answer other questions we make the following,

\begin{claim}\label{lem1}
Let $X$ be a map and $G \le$ Aut($X$). Suppose $V(X)$ has $m$ $G$-orbits. Then there exists a group $\widetilde{G} \le$ Aut($E$) such that $V(E)$ has $m$ $\widetilde{G}$-orbits.
\end{claim}
\noindent\emph{Proof of Claim \ref{lem1} :}
%\begin{proof}
Let $X$ be a semi-equivelar toroidal map of vertex type $[p_1^{n_1} \dots p_k^{n_k}]$. By proposition \ref{propo-1} we get $X = E_j/K$ for some discrete subgroup of Aut($E_j$) where $E_j$ is Archimedean tilling of $\mathbb{R}^2$. Let $\mathcal{O}_1, \dots \mathcal{O}_m$ be $G$-orbits of $V(X)$. Let $p:E_j\to X$ be the covering map. Then $\{p^{-1}(\mathcal{O}_i) | i = 1,2, \dots m\}$ be a partition of $V(E_j)$. 
%Let $\alpha \in G$.
%Then $\alpha: V(X) \to V(X)$ be a bijection. If we cut the torus $X$ then by that tiling of the flat torus we can fill up the whole plane. Now, suppose $v \in V(E_j)$. Then $v$ belongs to some copy of flat torus in $E_j.$ Apply $\alpha$ on that copy. Call this new map on $V(E_j)$ be $\widetilde{\alpha}$. Every point in $E_j$ can be represented as convex combination of some elements of $V(E_j)$. Using this extend $\widetilde{\alpha}$  to all over $E_j$ by defining $\widetilde{\alpha}(\sum_{i=1}^nt_iv_i) = \sum_{i=1}^nt_i\widetilde{\alpha}(v_i)$. Call that new map also $\widetilde{\alpha}$. Then $\widetilde{\alpha} \in Aut(E_j)$.
Now, consider $\widetilde{G} = {\rm Nor}_{Aut(E_j)}(K)$. 
%\langle \widetilde{\alpha},\alpha_j,\beta_j | \alpha \in G \rangle$ where $\alpha_j : z \mapsto z+A_j$ and $\beta_j: z \mapsto z + B_j$.
$V(E_j)$ forms $m$ $\widetilde{G}$-orbits. This proves our Claim \ref{lem1}.
%\end{proof}

\texttt{Case 3.}
Suppose, $X$ be a $3$-orbital toroidal map and $Y$ be a $2$-orbital cover of $X$. Then $X=E_7/\Kc_7$ and $Y=E_7/\mathcal{L}_7$ say. Let $G = \langle \alpha_7, \beta_7, \chi_7, \lambda_7 \rangle \le Aut(E_7)$ where $\chi_7$ and $\lambda_7$ be the maps obtained by taking reflection of $V(E_7)$ about the origin and the line $OA_7$ respectively.
Then $V(E_7)$ has $3$ $G$ orbits. $ \Kc_7 \trianglelefteq G.$
Let $H:= G/\Kc_7$. $H$ acts on $V(X)$ and gives $3$ orbits. 
Now, by Claim \ref{lem1} there exists $\Tilde{G} \le {\rm Aut}(E_7)$ containing $G $ such that $V(E_7)$ $2$ $\Tilde{G}$-orbits. But one can check that if we add any type of rotation or reflection in $G$ then number of $G$ orbits on $V(E_7)$ is either $1$ or $3$ not 2.  
Therefore there does not exists a 2-orbital cover of a 3-orbital map.
\smallskip

In other cases no such cover exists and it follows from Theorem \ref{thm-main0}.
\end{proof}
\begin{proof}[Proof of Theorem \ref{thm-main0}]
Suppose there exists a toroidal map $X$ of vertex type $[4^1,6^1,12^1]$ with $5$ Aut($X$)-orbits.  
By Claim \ref{lem1} there exists a group $\widetilde{G} \leq Aut(E_7)$ such that the action of $\widetilde{G}$ on $V(E_7)$ produces $5$ orbits. 
We know the isometries of the plane they are translations, rotations, reflections, and glide reflections. As we want finite number of orbits so $\widetilde{G}$ must contain translation group $H$. Now, let's see what are the possibility of $\widetilde{G}$.

\smallskip

\noindent{\textbf{Case 1.}} $\widetilde{G}=H$.

In this case number of orbits is $12$. Thus this case is not possible.

\smallskip

\noindent{\textbf{Case 2.}} $\widetilde{G}$ generated by $H$ and one rotation.

Only rotational symmetry of $E_7$ are rotation by angles $60,120,180$ degrees either clockwise or anticlockwise. If $\widetilde{G}$ contains $60$ degree, $120$ degree, $180$ degree rotation then number of orbits are $2, 3, 6$, respectively. None of them giving $5$ orbit.  Thus this case is also not possible.

\smallskip

\noindent{\textbf{Case 3.}} $\widetilde{G}$ generated by $H$ and one reflection.

Only reflectional symmetry of $E_7$ are reflection with respect to lines passing through origin and parallel to $A_7, B_7, F_7, \frac{A_7+B_7}{2}, \frac{B_7+F_7}{2}, \frac{F_7-A_7}{2}.$ Each of these cases number of orbits are $6$. Thus this case is also not possible.

\smallskip

\noindent{\textbf{Case 4.}} $\widetilde{G}$ generated by $H$ and two reflection.

Let $S_1 = \{R_{A_7}, R_{B_7}, R_{C_7}\}$ and $S_2 = \{R_{\frac{A_7+B_7}{2}}, R_{\frac{B_7+F_7}{2}}, R_{\frac{F_7-A_7}{2}}\}$ where $R_P$ denotes reflection about the line passing through origin and parallel to the vector $\overrightarrow{OP}$. Clearly one can check that if two reflections in $\widetilde{G}$ belongs to $S_1$ or $S_2$ then number of orbits is two and if one reflection belongs to $S_1$ and one in $S_2$ then number of orbits is either $3$ or $1$. Thus this case is also not possible.

\smallskip

\noindent{\textbf{Case 5.}} $\widetilde{G}$ generated by $H$ and one reflection and one rotation.

Clearly one can check that number of orbits in this case are either $1$ or $2$ or $3$. So this case is also not possible. Observe that these symmetries does not giving 4 orbits also.

Hence, there does not exist a subgroup $\widetilde{G}$ of Aut($E_7$) such that $E_7$ has 4 or 5 $\widetilde{G}$ orbits. 

Maps of type $[3^4, 6^1]$ have at least 1 vertex orbit, and
can have at most 3 vertex orbits. The only possible extra symmetries in this case are rotations,
since all maps of this kind are chiral (not invariant under orientation
reversing automorphisms). Furthermore, all half-turns in the symmetry group of the Archimedean tiling $[3^4, 6^1]$ normalize any translation subgroup, and so they are symmetries of every map on the torus with this type. It follows that only 3-fold and 6-fold rotations can add more
symmetry to a map on the torus with this type. Both of them collapse the three vertex orbits into only one, not two. Hence, there does not exists a 2-orbital map of type $[3^4,6^1]$ on torus.
\end{proof}

\begin{proof}[Proof of Theorem \ref{thm-main1}]
The proof follows from Lemmas  \ref{36}, \ref{3636},  \ref{312},  \ref{4612}.
\end{proof}

Now, we discuss the number of sheets of the cover $Y\longrightarrow X$ obtained in above lemmas in the next claim.
\begin{claim}\label{sheet} 
$Y$ is a $(m_1n_2-m_2n_1)$ sheeted covering of $X$.
\end{claim}
\noindent\emph{Proof of Claim \ref{sheet} :}
To proof this we need following two results from the theory of covering spaces.

\begin{result}\label{result1} (\cite{AH2002})
Let $p:(\widetilde{X},\widetilde{x_0})\to (X,x_0)$ be a path-connected covering space of the path-connected, locally path-connected space $X$, and let $H$ be the subgroup $p_*(\pi_1(\widetilde{X},\widetilde{x_0})) \subset \pi_1(X,x_0).$ Then,
\begin{enumerate}
    \item[1.] This covering space is normal if and only if $H$ is a normal subgroup of $\pi_1(X,x_0)$
    \item[2.] $G(\widetilde{X})$(the group of deck transformation of the covering $\widetilde{X}\to X$) is isomorphic to $N(H)/H$ where $N(H)$ is the normalizer of $H$ in $\pi_1(X,x_0)$.
\end{enumerate}
In particular, $G(\widetilde{X})$ is isomorphic to $\pi_1(X,x_0)/H$ if $\widetilde{X}$ is a normal covering. Hence, for universal cover $\widetilde{X}\to X$ we have $G(\widetilde{X}) \simeq \pi_1(X)$.
\end{result} 

\begin{result}\label{result2} (\cite{AH2002})
The number of sheets of a covering space $p:(\widetilde{X},\widetilde{x_0})\to (X,x_0)$ with $X$ and $\widetilde{X}$ path-connected equals the index of $p_*(\pi_1(\widetilde{X},\widetilde{x_0}))$ in $\pi_1(X,x_0)$.
\end{result}

In our case applying Result \ref{result1} for the covering $E_i\to E_i/\Kc_i$ we get $\pi_1(E_i/\Kc_i) = \Kc_i$. For the covering $E_i \to E_i/\mathcal{L}_i$ we get $\pi_1(E_i/\mathcal{L}_i) = \mathcal{L}_i$. Thus applying Result \ref{result2} we get number of sheets of $Y$ over $X$ is $=n:=[\Kc_i:\mathcal{L}_i] = (m_1n_2-m_2n_1)$ for all $i = 3,4,5,6, 7$.  This proves our Claim \ref{sheet}. 

Now, we prove  Theorem \ref{thm-main2}.

\begin{proof}[Proof of Theorem \ref{thm-main2}]
Let $X$ be a map of type $[p_1^{n_1}, \dots ,p_k^{n_k}]$. Then form Proposition \ref{datta2020} we get $X = E/K$ for some discrete subgroup $K$ of Aut($E$) where $E$ is the Archimedean tiling of the plane of same vertex type. Now, $Y$ covers $X$ if and only if  $Y = E/L$ for some subgroup $L$ of $K$ generated by $2$ translations corresponding to $2$ independent vectors. Let $K = \langle \gamma , \delta \rangle$. Now, consider $L_n = \langle \gamma^n, \delta \rangle$ and $Y_n = E/L_n$. then $Y_n$ covers $X$.
Number of sheets of the cover $Y_n \longrightarrow X$ is equal to $[K : L_n] = n.$ Hence, $Y_n$ is our required $n$-sheeted cover of $X$.
\end{proof}

Next, we prove  Theorems \ref{thm-main3} and \ref{thm-main4}. For that we need following definitions.

We call two maps are isomorphic if they are isomorphic as maps. Two maps are equal if the orbits of $\mathbb{R}^2$ under the action of corresponding groups are equal as sets.
Let $X$ and $K$ be as in the proof of Theorem \ref{thm-main2}. Let $n \in \mathbb{N}$. 
Let $Y =  E/L$ be $n$-sheeted cover of $X$. 
Let $L = \langle \omega_1, \omega_2\rangle$. Where $\omega_1, \omega_2 \in K = \langle \gamma, \delta \rangle$. Suppose $\omega_1 = \gamma^a \circ \delta^b$ and $\omega_2 = \gamma^c \circ \delta^d$ where $a,b,c,d \in \ZZ$. Define $M_Y = \begin{bmatrix} a & c\\b &d \end{bmatrix}$.
We represent $Y$ by the associated matrix $M_Y$. This matrix representation corresponding to a map is unique as $\gamma$ and $\delta$ are translations along two linearly independent vectors. Denote area of the torus $Y$ by $\Delta_Y$. As $Y$ is $n$-sheeted covering of $X$ so $\Delta_Y = n\Delta_X \implies$ area of the parallelogram spanned by $w_1$ and $w_2 = n \times$ area of the parallelogram spanned by $\gamma$ and $\delta$. That means $|det(M_Y)| = n$. Therefore for each $n$-sheeted covering, the associated matrix belongs to 
$$ \mathcal{S}:=\{ M \in GL(2,\ZZ): |det(M)| = n \}.$$ 
conversely, for every element of $\mathcal{S}$ we get a $n$-sheeted covering $Y$ of $X$ by associating $\begin{bmatrix} a & c\\b &d \end{bmatrix}$ to $E/\langle a\gamma + b\delta, c\gamma +d\delta\rangle$. So there is an one to one correspondence to $n$-sheeted covers of $X$ and $\mathcal{S}$. To proceed further we need following two lemmas.
\begin{lemma}\label{equl}
Let $Y_1$ and $Y_2$ be maps and $M_1$ and $M_2$ be associated matrix of them respectively. Then $Y_1 = Y_2$ if and only if  there exists an unimodular matrix (an integer matrix with determinant $1$ or $-1$) $U$ such that $M_1U = M_2.$  
\end{lemma}
\begin{proof}
Let $Y_1= E/\mathcal{L}_1$ and  $Y_2= E/\mathcal{L}_2$.  Suppose, $Y_1 = Y_2$. Let $i:Y_1\to Y_2$(the identity map) be an isomorphism. We can lift  $i$ to $\widetilde{i} \in Aut(E)$. Then $\widetilde{i}$ will take fundamental parallelogram of $Y_1$ to that of $Y_2$. Hence, the latices formed by $\mathcal{L}_1$ and $\mathcal{L}_2$ are same say $\Lambda$. $\widetilde{i}$ transforms $\Lambda$ to itself. Therefore from \cite{HW1979}(Theorem 32, Chapter 3) we get matrix of the transformation is unimodular. One side of our lemma follows from this.
\\
conversely, suppose $M_1U = M_2$ where $U$ is an unimodular matrix. Let $M_1 = (w_1~ w_2), M_2 = (w_1'~w_2') $ and $U = \begin{bmatrix} a & b\\c &d \end{bmatrix}$ where $w_i, w_i'$ are column vectors for $i=1,2$.
Therefore 
$$ M_2 = M_1U \implies (w_1'~ w_2') = (w_1~ w_2)\begin{bmatrix} a & b\\c & d \end{bmatrix} = (aw_1+cw_2~~bw_1+dw_2).$$
Now, suppose $\mathcal{L}_1 = \langle \alpha_1, \beta_1 \rangle$ and $\mathcal{L}_2 = \langle \alpha_2, \beta_2 \rangle$ and $A_i, B_i$ be the vectors by which $\alpha_i$ and $\beta_i$ translating the plane for $i=1,2$ and let $C$ and $D$ be the vectors corresponding to $\gamma$ and $\delta$.
Let $$A_1 = p_1C + q_1D, B_1 = s_1C + t_1D,$$ 
$$A_2 = p_2C + q_2D, B_2 = s_2C + t_2D.$$
Now, $w_1' = \begin{pmatrix} p_2 \\ q_2\end{pmatrix} = a\begin{pmatrix} p_1 \\ q_1\end{pmatrix} + c\begin{pmatrix} s_1 \\ t_1\end{pmatrix} = \begin{pmatrix} ap_1+cs_1 \\ aq_1+ct_1\end{pmatrix}$.
Therefore 
\begin{equation*}
    \begin{split}
        A_2 &= (ap_1 + cs_1)C + (aq_1+ct_1)D \\&= a(p_1C + q_1D) + c(s_1C + t_1D) \\&= aA_1 + cB_1
    \end{split}
\end{equation*}
Hence, $\alpha_2 \in \mathcal{L}_1$. Similarly $\beta_2 \in \mathcal{L}_1$. Therefore $\mathcal{L}_2 \le \mathcal{L}_1$. Proceeding in the similar way and using the fact that $det(U) = \pm 1$ we get $\mathcal{L}_1 \le \mathcal{L}_2$.
Therefore $\mathcal{L}_1 = \mathcal{L}_2$. Thus $Y_1 = E/\mathcal{L}_1 = E/\mathcal{L}_2 = Y_2$. This completes the proof of Lemma \ref{equl}.
\end{proof}
\begin{lemma}\label{isomm}
Let $Y_1$ and $Y_2$ be two toroidal maps with associated matrix $M_1$ and $M_2$ respectively. Then $Y_1 \simeq Y_2$ if and only if  there exists $A \in G_0$ and $B\in GL(2,\ZZ)$ such that $M_1 = AM_2B$ where $G_0$ is group of rotations and reflections fixing the origin in $E$ where $E$ is the tiling of $\mathbb{R}^2$ of type same as type of $Y_i$ for $i=1,2$. 
\end{lemma}
\begin{proof}
Let $Y_1 \simeq Y_2$ and $\alpha : Y_1 \to Y_2$ be an isomprphism. Now, $\alpha$ can be extended to an automorphism of the covering plane $E$, call that extension be $\widetilde{\alpha}$. 
Clearly $\widetilde{\alpha}$ will take fundamental parallelogram of $Y_1$ to that of $Y_2$. Now, the only ways to transform one fundamental region to another are rotation, reflection and change of basis of $E$. Multiplication by an element of $GL(2,\ZZ) $ will take care of base change. Rotation, reflection or their composition will take care by multiplication by $A\in G_0$. Hence, we get $M_1 = AM_2B$.\\
conversely, let $M_1 = AM_2B$. $A \in G_0$ so the combinatorial type of the torus associated to the matrix $AM_2$ and $M_2$ are same. Geometrically multiplying by elements of $GL(2,\ZZ)$ corresponds to modifying the fundamental domain by changing the basis. Hence, this will not change the combinatorial type of the torus. Thus $Y_1 \simeq Y_2$. This completes the proof of Lemma \ref{isomm}.
\end{proof}
Now, define a relation on $\mathcal{S}$ by $A\sim B \iff A= BU$ for some unimodular matrix $U$. Clearly this is an equivalence relation. Consider $\mathcal{S}' = \mathcal{S}/\sim $. So by Lemma \ref{equl} we can conclude that there are $\#\mathcal{S}'$ many distinct $n$-sheeted cover of $X$ exist. Let's find this cardinality.
Now, for every $m\times n$ matrix $A$ with integer entries has an unique $m \times n$ matrix $H$, called hermite normal form of $A$, such that $H = AU$ for some unimodular matrix $U$. All elements of an equivalence class of $\mathcal{S}'$ has  same hermite normal form and we take this matrix in hermite normal form as representative of that equivalence class. Thus to find cardinality of $\mathcal{S}'$ it is enough to find number of distinct matrices $M$ which are in hermite normal form and has determinant $n$. We do not take the matrices with determinant $-n$ because by multiplying by the unimodular matrix  
$\begin{bmatrix} 0 & 1\\1 &0 \end{bmatrix}$ changes sign of the determinant.
As $M$ is in lower triangular form so take $M = \begin{bmatrix} a & 0\\b &d \end{bmatrix}$.
Then $det(M) = ad = n \implies a= n/d.$ By definition of hermite normal form $b\geq0$ and $b<d$ so $b$ has $d$ choices for each $d|n.$ Hence, there are precisely $\sigma(n):=\sum_{d|n}d$ many distinct $M$ possible. Thus $\#\mathcal{S}' = \sigma(n).$

Therefore there are $\sigma(n)$ many distinct $n$-sheeted cover exists for a given toroidal map.
However depending on the group $G_0$ we are counting isomorphic maps more than once in the collection of distinct $n$-sheeted covers. Let $D_n$ denotes the symmetry group of a regular $n$-gon.
Then $G_0 = D_6$ for tilings of the plane of type $[3^6], [6^3],$\\ $[3^1, 4^1, 6^1, 4^1], [3^1, 12^2], [4^1, 6^1,12^1], [3^1, 6^1, 3^1, 6^1]$, $G_0 = \mathbb{Z}/6\mathbb{Z}$ for tilings of the plane of type $[3^4,6^1]$, $G_0 = D_4$ for tilings of the plane of type $[4^4], [3^2, 4^1, 3^1, 4^1], [4^1, 8^2]$ and $G_0 = \mathbb{Z}/2\mathbb{Z}$ for tilings of the plane of type $[3^3, 4^2]$.
For a map $X$ let us denote $\mathcal{S}_X$ be the stabilizer of a vertex $v$ of $X$.
$\mathcal{S}_X := \{\phi \in {\rm Aut}(X) \mid \phi(v) = v \}.$
Then $\mathcal{S}_X \leq G_0$ for corresponding tiling.
For $G_0 = D_6$ possible orders of $\mathcal{S}_X$ are $2,4,6 ~{\rm and}~ 12$. Order $1$ is not possible because point reflection is always present in ${\rm Aut}(X)$ and it has order $2$. 
If $\mathcal{S}_X$ has order $2$ it is counted $6$ times (just by $60$ degree rotations). If $\mathcal{S}_X$ has order $4$ then it is counted $3$ times. If $\mathcal{S}_X$ has order $6$ it is counted $2$ times and if $\mathcal{S}_X$ has order $12$ then it is counted only once.
In general, if $\mathcal{S}_X$ has order $d$ then $X$ is counted $12/d$ times. This is caused because the choice of basis of the map is not unique, by applying different symmetries if the tiling on the basis we get other representation of the same map.
Similarly, for those tilings which have stabilizer $D_4$, the stabilizer of a corresponding map $X$ will be subgroup of $D_4$. So possible orders of $\mathcal{S}_X$ will be $2,4 ~{\rm and}~ 8$. If $\mathcal{S}_X$ has order $d$ then it is counted $8/d$ times.
And for those tilings which have stabilizer $\mathbb{Z}/4\mathbb{Z}$ and $\mathbb{Z}/6\mathbb{Z}$, the stabilizer of a corresponding map $X$ will be subgroup of $\mathbb{Z}/4\mathbb{Z}$ and $\mathbb{Z}/6\mathbb{Z}$ respectively. So possible orders of $\mathcal{S}_X$ will be $2, 4$ or $2,6$. If $\mathcal{S}_X$ has order $d$ then it is counted $4/d$ or $6/d$ times respectively. For the latter case order of $\mathcal{S}_X$ cannot be $3$ since in that case the group will be generated by $120^{\circ}$ rotation and presence of $120^{\circ}$ rotation will ensure the presence of $60^{\circ}$ rotation. A function  $f \colon \mathbb{N} \to \mathbb{N}$ is called {\em multiplicative} if for any two co-prime positive integers $m$ and $n$, $f(mn)=f(m)f(n)$.
Now, we are ready to calculate total number of $n$-sheeted covers of a given map.

\begin{proof}[Proof of Theorem \ref{thm-main3}]
Let $X$ be map of type $Z$ and $Y$ be an $n$-orbital cover of $X$.
suppose $Z=[3^4,6^1]$. The stabilizer group of this tiling is generated by $60^\circ$ rotation. Therefore for a map $X$, $\mathcal{S}_X$ has orders either 2 or 6. If $\mathcal{S}_X$  has order 2 then we are counting it thrice and if it has order $6$ then we are counting it only once in the collection of distinct $n$-sheeted covers.
Let $\mathcal{M}_i(n)$ denotes number of $n$-sheeted covers up to isomorphism with order of the stabilizer $i$ for $i=2$ and $6$. Then $3\mathcal{M}_2(n) + \mathcal{M}_6(n) = \sigma(n)$. Now, $\mathcal{M}_6(n)$ is also number of $n$-sheeted covers having $60^\circ$ rotation in the stabilizer. For this counting we have the following claim. First define,
$$\rho_{i,j}(n) := {\rm Number~of~solutions~of ~} x^2+ix+j\equiv 0 \pmod n.$$
\begin{claim}\label{cm-1}
If $X$ be an $n$-sheeted toroidal cover of a given map $X_0$ of type $[3^4,6^1]$, then the total number of maps having $60^\circ$ rotation in the automorphism group of $X$ is equal to $$f_1(n)=\left\{
	\begin{array}{ll}
		0  & \mbox{, if } m_j\equiv1\pmod2 \mbox{ for some } j\in\{0,1,2,\dots,n_2\} \\
		\prod_{i=1}^{n_1}(k_i+1) & \mbox{, otherwise.}
	\end{array}
\right.$$
where $n=2^{m_0} \cdot 3^{k_0} \cdot \prod_{i=1}^{n_1}p_i^{k_i} \cdot \prod_{j=1}^{n_2}q_j^{m_j}$ with $p_i$ and $q_j$ are primes such that $p_i\equiv1\pmod3$ for $i\in \{0,1,\dots,n_1\}$ and $q_j\equiv2\pmod3$ for $j\in \{0,1,\dots,n_2\}$.
\end{claim}
\noindent\emph{Proof of Claim \ref{cm-1} :}
Let $X$ be an $n$-sheeted toroidal cover of $X_0$ of type $[3^4,6^1]$ with associated matrix $\begin{bmatrix} a & 0\\ b & d\end{bmatrix}$. We have taken this matrix representation with respect to basis vectors of $X_0$. So, $ad=n$. Then $X=\frac{E}{\mathcal{K}}$ with $\mathcal{K}=\langle w_1,\,w_2\rangle$ where $w_1:z\mapsto z+(aA_5+bB_5)$ and $w_2: z\mapsto z+dB_5$ (See Fig. \ref{fig:Archi}). 
Let $\rho$ denotes the map obtained by taking $60^{\circ}$ rotation on $E_5$.
$\rho$ sends $A_5$ to $B_5$ and $B_5$ to $B_5-A_5$.

%\begin{proof}
We look for the conditions on $a,\,b$ and $d$ so that $\rho\in{\rm Nor}(\mathcal{K})$. For that, it is enough to check if $\rho^{-1}w_1\rho$, $\rho^{-1}w_2\rho$ belong $\mathcal{K}$. 
We know that conjugation of a translation by a rotation is again a translation by the rotated vector. So, $\rho^{-1}w_i\rho$  is translation by the vector $\rho \circ w_i(0).$  Now, $\rho \circ w_1(0)=\rho(aA_5+bB_5)=a\rho(A_5)+b\rho(B_5)=aB_5+b(B_5-A_5)=-bA_5+(a+b)B_5$ and  $\rho \circ w_2(0)=\rho(cA_5+dB_5)=c\rho(A_5)+d\rho(B_5)=cB_5+d(B_5-A_5)=-dA_5+(c+d)B_5$. Now, $\rho(w_1)$ and $\rho(w_2)$ belong to the lattice of $\mathcal{K}$ provided that there exists integers $m_1,\,m_2,\,m_3,\, m_4$ such that $-bA_5+(a+b)B_5=m_1(aA_5+bB_5)+m_2dB_5$ and $-dA_5+dB_5=m_3(aA_5+bB_5)+m_4dB_5.$ Since, $A_5$ and $B_5$ are linearly independent, we have a system of linear equations, $m_1a=-b$, $m_1b+m_2d=a+b$, $m_3a=-d$ and $m_3b+m_4d=d$. Solving these equations, we get, $m_1=-\frac{b}{a}$, $m_2=\frac{a^2+ab+b^2}{ad}$, $m_3=-\frac{d}{a}$ and $m_4=\frac{a+b}{a}$. Since $m_1,\,m_2,\,m_3$ and $m_4$ are integers, we must have $a\mid b$, $a\mid d$ and $ad\mid(a^2+ab+b^2)$. Therefore, $\rho \in {\rm Aut}(X)$ if and only if {\rm(i)} $a\mid b$, {\rm(ii)} $a\mid d$ and {\rm(iii)} $ad\mid(a^2+ab+b^2)$.

Now, $a\mid b$ and $a \mid d$ implies $b=ax$ and $d=ay$ where $0\leq x<y$. Using the last condition, we get $a.ay\mid (a^2+a.ax+ax.ax)$, or, $y\mid (1+x+x^2)$ where $y=\frac{d^2}{n}$. This implies $ 1+x+x^2$ has a solution in $\mathbb{Z}_{\frac{d^2}{n}}$. So, the total number of maps is given by the number of solutions of the polynomial $1+x+x^2$ in $\mathbb{Z}_{\frac{d^2}{n}}$ satisfying $n\mid d^2$ for every divisor $d$ of $n$. Hence, we have 
$$f_1(n):=\sum_{\substack{d \mid n \\n \mid d^2}}\rho_{1,1}\left(\frac{d^2}{n}\right).$$
from \cite[Theorem-122, Chapter 8]{HW1979} we have the following,
\begin{result}\label{result--1}
The number of roots of $f(x) \equiv 0 \pmod n$ is the product of the number of roots of separate congruences $f(x) \equiv 0 \pmod{p_i^{t_i}}$ for $i=1,2,3,\dots,k$.
\end{result}
From the above result we have $\rho_{1,1}$ is a multiplicative function. As a consequence of this we have $f_1$ is also multiplicative. Thus it is enough to find values of $f_1$ for prime powers.

$f_1(p^k)=\sum_{\substack{d \mid p^k \\p^k \mid d^2}}\rho_{1,1}\left(\frac{d^2}{p^k}\right) = 
\sum_{\substack{p^l \mid p^k \\p^k \mid p^{2l}}}\rho_{1,1}\left(\frac{p^{2l}}{p^k}\right)=
\sum_{\frac{k}{2}\le l \le k}\rho_{1,1}\left(p^{2l-k}\right)=
\sum_{l=\lceil \frac{k}{2} \rceil}^k\rho_{1,1}\left(p^{2l-k}\right)$
Therefore, we have for all prime $p$,
$$f_1(p^k) =\left\{
	\begin{array}{ll}
		\sum_{l= \frac{k}{2} }^k\rho_{1,1}\left(p^{2l-k}\right)  & \mbox{, if  k is even} \\
		\sum_{l= \frac{k+1}{2} }^k\rho_{1,1}\left(p^{2l-k}\right) & \mbox{, if  k is odd}
	\end{array}
          \right.$$

From theory of quadratic residues \cite{HW1979} we have $1+x+x^2 \equiv 0 \pmod {p^i}$ has 2 distinct solutions when $p \equiv 1 \pmod  3$ provided $i>0$. If $p \equiv 2 \pmod 3$ then $1+x+x^2 \equiv 0 \pmod {p^i}$ has no solution provided $i>0$. If $i=0$ then it has unique solution in both the cases.
Putting these values in the above formula we have,
$$f_1(p^k) =\left\{
	\begin{array}{lll}
		k+1  &, p \equiv 1 \pmod  3 \\
		0 &, p \equiv 5 \pmod  3 ~{\rm and}~ 2 \nmid k \\
		1 &, p \equiv 5 \pmod  3 ~{\rm and}~ 2 \mid k
	\end{array}
          \right.$$
          
Now,
observe that $\rho_{1,1}(2^k)=0$ for all $k\in \mathbb{N}$ since $1+x+x^2$ is odd for every $x\in \mathbb{N}\cup\{0\}$ and $2^k\nmid (1+x+x^2)$.
$$f_1(2^k) =\left\{
	\begin{array}{ll}
		\sum_{l= \frac{k}{2} }^k\rho_{1,1}\left(2^{2l-k}\right) = 1 & \mbox{, if  k is even} \\
		\sum_{l= \frac{k+1}{2} }^k\rho_{1,1}\left(2^{2l-k}\right) = 0  & \mbox{, if  k is odd}
	\end{array}
          \right. $$
$f_1(3^k)=1$ for all $k\in\mathbb{N}$.
$\rho_{1,1}(3)=1$ since $1$ is the only solution of the equation $1+x+x^2$ in $\mathbb{Z}_3$.
We want to show that $\rho_{1,1}(3^k)=0$ for all $k\geq 2$. When $x\equiv 0 \pmod 3$ then $1+x+x^2\equiv 1 \pmod 3$. When $x=3l+1$ where $l\in\mathbb{N}$, $1+x+x^2=3(1+3l+3l^2)$. When $x\equiv 2 \pmod 3$ then $1+x+x^2\equiv 1 \pmod 3$. So, $3^k\nmid(1+x+x^2)$ for all $x\in \mathbb{N}$. Hence, $\rho_{1,1}(3^k)=0$.
Therefore, 
$$f_1(3^k) =\left\{
	\begin{array}{ll}
		\sum_{l= \frac{k}{2} }^k\rho_{1,1}\left(3^{2l-k}\right) = 1 & \mbox{, if  k is even} \\
		\sum_{l= \frac{k+1}{2} }^k\rho_{1,1}\left(3^{2l-k}\right) = 1  & \mbox{, if  k is odd}
	\end{array}
          \right. $$

Now, by fundamental theorem of arithmetic any integer $n$ can be written as  $n=2^{m_0} \cdot 3^{k_0} \cdot \prod_{i=1}^{n_1}p_i^{k_i} \cdot \prod_{j=1}^{n_2}q_j^{m_j}$ with $p_i$ and $q_j$ are primes such that $p_i\equiv1\pmod3$ for $i\in \{0,1,\dots,n_1\}$ and $q_j\equiv2\pmod3$ for $j\in \{0,1,\dots,n_2\}$, using multiplicativity of $f $ we get,
$$f_1(n):=\left\{
	\begin{array}{ll}
		0  & \mbox{, if } m_j\equiv1\pmod2 \mbox{ for some } j\in\{0,1,2,\dots,n_2\} \\
		\prod_{i=1}^{n_1}(k_i+1) & \mbox{, otherwise,}
	\end{array}
      \right. $$
This completes the proof of Claim \ref{cm-1}.

The same conclusion will be true for maps of type $[3^1,4^1,6^1,4^1],~[3^1,12^2],~[3^1,6^1,3^1,6^1]$ and $[4^1,6^1,12^1]$.
\smallskip

Hence, $\mathcal{M}_6(n) = f_1(n)$. Therefore, $\mathcal{M}_2(n) = \frac{\sigma(n)-f_1(n)}{3}$.
Thus total number of $n$-sheeted covers up to isomorphism is $\mathcal{M}_2(n)+\mathcal{M}_6(n)=\frac{\sigma(n)+2f_1(n)}{3}$.

Next, suppose $Z=[3^1,4^1,6^1,4^1]$ or $[3^1,6^1,3^1,6^1]$ or $[4^1,6^1,12^1]$ or $[3^1,12^2]$. The stabilizer group for these tilings are $D_6$. Therefore, for a map $X$, order of $\mathcal{S}_X \in \{2,4,6,12\}$.   Let $\mathcal{N}_i(n)$ denotes number of $n$-sheeted covers up to isomorphism with order of the stabilizer $i$ for $i=2,4,6$ and 12 respectively. Then,
\begin{equation}\label{eqn}
    6\mathcal{N}_2(n) + 3\mathcal{N}_4(n) + 2\mathcal{N}_6(n) + \mathcal{N}_{12}(n) = \sigma(n).
\end{equation}
To get explicit expression of $\mathcal{N}_i(n)$ we are making following claim.
Let $R_1, R_2, R_3$ denotes the line which bisects the angle between $OA$ and $OB$, $OB$ and $OF$, $OF$ and $O(-A)$ where $F=B-A$ and $A=A_j, B=B_j$ for $j=3,4,7$ or $6$. Let $r_i$ denoted the reflection about $R_i$ for $i=1,2,3$ respectively.  
\begin{claim}\label{cm--3}
If $X$ be a $n$-sheeted toroidal cover of a given map $X_0$, then the total number of maps having $r_i$, for $i=1,2,3$, in the automorphism group of $X$ is equal to $$f_3(n)=\left\{
	\begin{array}{ll}
		\prod_{i=1}^{n_1}(k_i+1)  & \mbox{, if } k_0=0 \\
		(2k_0-1)\prod_{i=1}^{n_1}(k_i+1) & \mbox{, otherwise.}
	\end{array}
\right.$$ 
where $n=2^{k_0}\cdot \prod_{i=1}^{n_1}p_i^{k_i}$ such that $p_i$ is any odd prime for $i\in \{0,1,\dots,n_1\}$.
\end{claim}
\noindent\emph{Proof of Claim \ref{cm--3} :}
Let $X, \mathcal{K}, w_1, w_2$ be as in proof of Claim \ref{cm-1}.

%\begin{claim}\label{clm-1}

%\end{claim}
%\begin{proof}

We know that conjugation of a translation by a reflection is again a translation by the reflected vector. So, $r_1^{-1}\circ w_i\circ r_1$ is translation by $r_1 \circ w_i(0)$. Now, $=r_1\circ w_1(0)=r_1(aA+bB)=ar_1(A)+br_1(B)=aB+b(A-B)=bA+(a-b)B$ and $r_1\circ w_2(0)=r_1(dB)=dr_1(B)=d(A-B)=dA-dB$. By applying similar method used in Claim \ref{cm-1}, we get $r_1 \circ w_1,\,r_1 \circ w_2\in \mathbb{Z}(aA+bB)+\mathbb{Z}dB$ if and only if $a\mid b$, $a\mid d$ and $ad\mid (b^2+2ab)$. Thus, $r_1\in$ {\rm Nor}$(\mathcal{K})$ if and only if {\rm (i)} $a\mid b$, {\rm (ii)} $a\mid d$ and {\rm (iii)} $ad\mid (b^2+2ab)$.

Using (i) and (ii) we get $b=ax$ and $d=ay$ where $0\leq x\leq y$ and using (iii), we get $y\mid (x^2+2x)$ where $y=\frac{d^2}{n}$. Hence, the total number of maps having $r_1 \in $ Aut($X$) is given by the number of solutions of the polynomial $x^2+2x$ in $\mathbb{Z}_{\frac{d^2}{n}}$ satisfying $n\mid d^2$ for every divisor $d$ of $n$. % Similarly for $r_2$ we get the polynomial $x^2-1$. 
Hence, we get,
$$f_3(n)=\sum_{\substack{d \mid n \\n \mid d^2}}\rho_{2,0}\left(\frac{d^2}{n}\right).$$
By similar reasons $f_3$ is also multiplicative and we have 
$$f_3(p^k) =\left\{
	\begin{array}{ll}
		\sum_{l= \frac{k}{2} }^k\rho_{2,0}\left(p^{2l-k}\right)  & \mbox{, if  k is even} \\
		\sum_{l= \frac{k+1}{2} }^k\rho_{2,0}\left(p^{2l-k}\right) & \mbox{, if  k is odd}
	\end{array}
          \right.$$
Observe that,
$\rho_{2,0}(1)=1$ since $x=0$ is the solution for $x^2+2x\equiv0\pmod1$.\\
$\rho_{2,0}(2)=1$ since $0$ is the only solution of $0$ in $\mathbb{Z}_2$.
$\rho_{2,0}(4)=2$ since $0$ and $2$ are the solutions of $x^2+2x=0$ in $\mathbb{Z}_4$.
Now, $x^2+2x\equiv 0\pmod{2^k}\implies (x+1)^2\equiv1\pmod{2^k}$. Putting $y=x+1$, we get the congruence $y^2\equiv1\pmod{2^k}$. We know that the congruence $y^2\equiv1\pmod{2^k}$ has exactly four incongruent solutions. Hence, $\rho_{2,0}(2^k)=4$.
Therefore,
$$f_3(2^k) =\left\{
	\begin{array}{ll}
		\sum_{l= \frac{k}{2} }^k\rho_{2,0}\left(2^{2l-k}\right) = 1+2+4+\dots+4=4(\frac{k}{2}-1)+3=2k-1  & \mbox{, if  k is even} \\
		\sum_{l= \frac{k+1}{2} }^k\rho_{2,0}\left(2^{2l-k}\right)=1+4+4+\dots+4=4(\frac{k-1}{2})+1=2k-1 & \mbox{, if  k is odd}
	\end{array}
          \right.$$

The congruence $x^2+2x\equiv0\pmod{p}$ has solution if and only if $y^2\equiv4\pmod{p}$ has a solution. Since, $p$ is odd we have $\left(\frac{4}{p}\right)=1$ that is 4 is a quadratic residue $\mod p$. Hence, $x^2+2x\equiv0\pmod{p}$ has exactly two solutions for all odd prime $p$. This implies $x^2+2x\equiv0\pmod{p^k}$ has $2$ solutions for all $k\geq1$. So, $\rho(p^k)=2$ for all $k\geq1$.
$$f_3(p^k) =\left\{
	\begin{array}{ll}
		\sum_{l= \frac{k}{2} }^k\rho_{2,0}\left(p^{2l-k}\right) = 1+2+2+\dots+2=2(\frac{k}{2})+1=k+1  & \mbox{, if  k is even} \\
		\sum_{l= \frac{k+1}{2} }^k\rho_{2,0}\left(p^{2l-k}\right)=2+2+\dots+2=4(\frac{k+1}{2})=k+1 & \mbox{, if  k is odd}
	\end{array}
          \right.$$
Therefore by multiplicative property of $f_3$ we have,
$$f_3(n):=\left\{
	 \begin{array}{ll}
		\prod_{i=1}^{n_1}(k_i+1)  & \mbox{, if } k_0=0 \\
		(2k_0-1)\prod_{i=1}^{n_1}(k_i+1) & \mbox{, otherwise.}
	 \end{array}
     \right.$$
where $n=2^{k_0}\cdot \prod_{i=1}^{n_1}p_i^{k_i}$ such that $p_i$ is any odd prime for $i\in \{0,1,\dots,n_1\}$. Similarly for $r_2$ and $r_3$ we will get the same function. This completes the proof of Claim \ref{cm--3}.
\smallskip

Let number of distinct $n$-sheeted cover $Y$ such that $\mathcal{S}_Y=\mathcal{S}_{E_3}$ is given by $A(n)$. Then $\mathcal{N}_{12}(n) = A(n)$, since we are counting these maps only once
in the collection of all distinct $n$-sheeted covers.
Let $D_i(n)$ denotes number of distinct $n$-sheeted covers with stabilizer is generated by $r_i$ and the point reflection. Then $D_i(n) = f_3(n)-A(n)$ for $i=1,2,3$. Let $D_4(n)$ denotes the number of distinct $n$-covers with stabilizer generated by $60^\circ$ rotation. Then, $D_4(n) = f_1(n)-A(n)$.
Therefore, $\mathcal{N}_6(n)=D_4(n)/2$ and $\mathcal{N}_4(n)=(D_1(n)+D_2(n)+D_3(n))/3$. Now, $\mathcal{N}_2(n)$ can be determined uniquely from equation (\ref{eqn}).
Hence, total number of $n$-sheeted covers is given by $\mathcal{N}_2(n)+\mathcal{N}_4(n)+\mathcal{N}_6(n)+\mathcal{N}_{12}(n) = \frac{1}{6}(\sigma(n)+2f_1(n)+3f_3(n))$.

Suppose $Z= [3^2,4^1,3^1,4^1]$. Let $L$ denotes the line passing through $O$ and $A_1$ in $E_1$ (see Fig. \ref{fig:Archi}) and $R$ is the reflection about $L$. Let $T$ denoted the translation of $E_1$ which maps $O$ to the mid point of the line joining $u_{0,1}$ and $u_{-1,2}$. Then the glide reflection $g := T\circ R$ is a symmetry of $E_1$. Note that the translation $T$ and the reflection $R$ are not symmetries of the tiling. Let $\phi$ be the function obtained by rotating $E_1$ by 90 degree with respect to origin. There are two more reflection present in the tiling. One is reflection about the line joining $u_{1,-1}$ and $u_{0,0}$. Other is reflection about the line joining $u_{-1,0}$ and $u_{0,1}$. But presence of latter two reflection in the automorphism group of a map does not effect to the number of orbits. Thus we do not consider these two reflections for counting number of covers up to isomorphism.
Observe that if $g$ or $\phi$ present in Aut($Y$) then $Y$ is one orbital otherwise it is 2-orbital. To count number of covers up to isomorphism we make following claims.

\begin{claim}\label{cm---2}
If $X$ be an $n$-sheeted toroidal cover of a given map $X_0$, then the total number of maps having $\phi$ in the automorphism group of X is equal to 
\[f_2(n)=\left\{
	\begin{array}{ll}
		0  & \mbox{, if } m_j\equiv1\pmod2 \mbox{ for some } j\in\{0,1,2,\dots,n_2\} \\
		\prod_{i=1}^{n_1}(k_i+1) & \mbox{, otherwise.}
	\end{array}
\right.\] 
where $n=2^{m_0} \cdot 3^{k_0} \cdot \prod_{i=1}^{n_1}p_i^{k_i} \cdot \prod_{j=1}^{n_2}q_j^{m_j}$ such that $p_i$ and $q_j$ are primes with $p_i\equiv1\pmod4$ for $i\in \{0,1,\dots,n_1\}$ and $q_j\equiv3\pmod4$ for $j\in \{0,1,\dots,n_2\}$.
\end{claim}

\noindent\emph{Proof of Claim \ref{cm---2} :}
Let $X, \mathcal{K},w_1,w_2$ be as in Claim \ref{cm-1} but of type $[3^2,4^1,3^1,4^1]$. Now, $\phi$ sends $A_1$ to $B_1$ and $B_1$ to $-A_1$ (See Fig. \ref{fig:Archi}). 
Here, $\phi^{-1}w_1\phi$ and $\phi^{-1}w_2\phi$ are translation by vectors $\phi \circ w_1(0)$ and $\phi \circ w_2(0)$ respectively. Now, $\phi \circ w_1(0)=\phi(aA_1+bB_1)=a\phi(A_1)+b\phi(B_1)=aB_1+b(-A_1)=-bA_1+aB_1$ and $\phi \circ w_2(0)=\phi(dB_1)=d\phi(B_1)=d(-A_1)=-dA_1$.  These two vectors belong to the lattice of $\mathcal{K}$ if and only if 
%there exists integers $m_1,n_1,m_2,n_2$ such that $\phi \circ w_i(0) = m_iw_1(0) + n_iw_2(0)$ for $i=1,2$. Putting values of $w_i$'s in above we get a system of linear equation. Since solutions are integers we get the following conditions on entries of the matrix.
$a\mid b$, $a\mid d$ and $ad\mid(a^2+b^2)$.\\
Now, $a\mid b$ and $a \mid d$ implies $b=ax$ and $d=ay$ where $0\leq x<y$. Using the last condition, we get $a.ay\mid (a^2+ax.ax)$, or, $y\mid (1+x^2)$ where $y=\frac{d^2}{n}$. This implies $ 1+x^2$ has a solution in $\mathbb{Z}_{d^2/n}$. Proceeding similarly as in Claim \ref{cm-1} we get, 
$$f_2(n):=\sum_{\substack{d \mid n \\n \mid d^2}}\rho_{0,1}\left(\frac{d^2}{n}\right).$$
Similarly, as Claim \ref{cm-1} $f_2$ will be multiplicative as well and we have 
$$f_2(p^k) =\left\{
	\begin{array}{ll}
		\sum_{l= \frac{k}{2} }^k\rho_{0,1}\left(p^{2l-k}\right)  & \mbox{, if  k is even} \\
		\sum_{l= \frac{k+1}{2} }^k\rho_{0,1}\left(p^{2l-k}\right) & \mbox{, if  k is odd}
	\end{array}
          \right.$$

$\rho_{0,1}(1)=1$ since $x=0$ is the solution for $1+x^2\equiv0\pmod1$.
$\rho_{0,1}(2)=1$ since $1$ is the only solution of $1+x^2=0$ in $\mathbb{Z}_2$.
Now, $\rho_{0,1}(2^k)=0$ for all $k\in \mathbb{N}$ with $k\geq 2$. This is because when $x=2l$ for some $l\in\mathbb{N}$, $1+x^2=1+4l^2$. So, $2^k\nmid 1+x^2$ for all $k$. 
Let $n=2^k$. Then, $d=2^i$, $0\leq i\leq k$.
$$f_2(2^k) =\left\{
	\begin{array}{ll}
		\sum_{l= \frac{k}{2} }^k\rho_{0,1}\left(2^{2l-k}\right) = 1  & \mbox{, if  k is even} \\
		\sum_{l= \frac{k+1}{2} }^k\rho_{0,1}\left(2^{2l-k}\right)=1 & \mbox{, if  k is odd}
	\end{array}
          \right.$$
When $x=3l$ where $l\in\mathbb{N}\cup \{0\}$, $+x^2=1+9l^2$. When $x=3l+1$ where $l\in\mathbb{N}$, $1+x^2=9k^2+6k+2$. When $x=3l+2$ where $l\in\mathbb{N}$, $1+x^2=9k^2+12k+5$. So, $3^k\nmid(1+x^2)$ for all $x\in \mathbb{N}$. Hence, $\rho_{0,1}(3^k)=0$.
Let $n=3^k$. Then, $d=3^i$, $0\leq i\leq k$.
$$f_2(3^k)=
\left\{
	\begin{array}{ll}
		\sum_{l= \frac{k}{2} }^k\rho_{0,1}\left(2^{2l-k}\right)=1 & \mbox{, if  k is even} \\
		\sum_{l= \frac{k+1}{2} }^k\rho_{0,1}\left(2^{2l-k}\right)=0 & \mbox{, if  k is odd}
	\end{array}
\right.$$

If $p$ is any odd prime except $3$, then, $x^2+1 \equiv 0\pmod p$ has a solution if and only if $y^2\equiv -4\pmod p$ has a solution. 
Using Legendre symbols \cite[Chap. 9]{burton2018}, $\left(\frac{-4}{p}\right)=\left(\frac{-1}{p}\right)\left(\frac{4}{p}\right)=\left(\frac{-1}{p}\right)=\left\{
	\begin{array}{ll}
		1  & \mbox{if } p\equiv1\pmod4 \\
		-1 & \mbox{if } p\equiv3\pmod4
	\end{array}
\right.$.
Hence, we have $\rho_{0,1}(p)=
\left\{
	\begin{array}{ll}
		2  & \mbox{if } p\equiv1\pmod4 \\
		0 & \mbox{if } p\equiv3\pmod4
	\end{array}\right.$.
Putting these values in the expression of $f_2(p^k)$ we get,
$$ f_2(p^k)=
\left \{
	\begin{array}{lll}
	    k+1 & if p\equiv 1\pmod 4\\
		1  & \mbox{if } k\equiv 0\pmod2 ~\mbox{and}~ p\equiv 3\pmod 4 \\
		0  & \mbox{if } k\equiv 1\pmod2 ~\mbox{and}~ p\equiv 3\pmod 4
	\end{array}
\right .$$

Let $n=2^{k_0}3^{m_0}p_1^{k_1}p_2^{k_2}\dots p_{n_1}^{k_{n_1}}q_1^{m_1}q_2^{m_2}\dots q_{n_2}^{m_{n_2}}$ where $p_i$ and $q_j$ are primes with $p_i\equiv1\pmod4$ where $i\in \{0,1,\dots,n_1\}$ and $q_j\equiv3\pmod4$ where $j\in \{0,1,\dots,n_2\}$.\\
Then using multiplicavity of $f_2$ we get,
\[f_2(n)=\left\{
	\begin{array}{ll}
		0  & \mbox{, if } m_j\equiv1\pmod2 \mbox{ for some } j\in\{0,1,2,\dots,n_2\} \\
		\prod_{i=1}^{n_1}(k_i+1) & \mbox{, otherwise.}
	\end{array}
\right.\] 
This completes the proof of Claim \ref{cm---2}.

\begin{claim}\label{glide}
Let $X$ be a $n$-sheeted toroidal cover of a given map of type $[3^2,4^1,3^1,4^1]$. Then number of distinct $n$-sheeted covers having glide in its automorphism group is 
$$g(n) := \sum_{d \mid n,~2 \mid d}2 + \sum_{d \mid n,~2 \nmid d}1.$$
\end{claim}

\noindent\emph{Proof of Claim \ref{glide} :}
Corresponding matrix for glide reflection is $\begin{bmatrix}
-1 & 0 \\ 0 & 1 \end{bmatrix}$. A map represented by $\begin{bmatrix}
a&0\\b&d
\end{bmatrix}$ contains $g$ in its automorphism group if and only if $d \mid 2b$. Since, $0\leq b <d$ if $d$ is even then there are only 2 possible values of $b$, they are $b=0$ and $\frac{d}{2}$. If $d$ is odd then 0 is only possible value of $b$. Hence, number of such matrix is $\sum_{\substack{d \mid n \\2 \mid d}}2 + \sum_{\substack{d \mid n \\2 \nmid d}}1$. %This completes the proof of Claim \ref{glide}.

\begin{claim}\label{clm-h}
Let $X$ be a $n$-sheeted toroidal cover of a given map of type $[3^2,4^1,3^1,4^1]$ having $\phi$ and glide reflection in its stabilizer. Then number of such distinct maps will be 
$$h(n):=\sum_{\substack{d \mid n \\n \mid d^2}}\rho_7\left(\frac{d^2}{n}\right).$$
\end{claim}

\noindent\emph{Proof of Claim \ref{clm-h} :}
$\phi$ and glide both will present in the automorphism group of a map represented by $\begin{bmatrix}
a&0\\b&d
\end{bmatrix}$ if and only if $a\mid b$, $a\mid d$, $ad\mid a^2+b^2$ and $d \mid 2b$. Let
$\rho_7(n):=\#\{x\in \mathbb{Z}_n:\, x^2+1=0\,{\rm and}\,2x=0\}$. Therefore, $h(n):=\sum_{\substack{d \mid n \\n \mid d^2}}\rho_7\left(\frac{d^2}{n}\right).$ 
%By similar reasons like Claim \ref{cm-1}, $h$ is multiplicitive function. 
%The explicit value of $h(n)$ is not required so we skip it.

\begin{claim}\label{alpha}
Let $X$ be a $n$-sheeted toroidal cover of a given map of type $[3^2,4^1,3^1,4^1]$  such that after applying glide reflection and $90^{\circ}$ rotation on $X$ we get same map and Aut($X$) does not contain 90$^{\circ}$ or glide. Then number of distinct such maps are,
$$\alpha(n) = f_3(n)-h(n).$$
\end{claim}
\noindent\emph{Proof of Claim \ref{alpha} :} 
Let $X$ be represented by the matrix $M=\begin{bmatrix}
a&0\\b&d
\end{bmatrix}$. Let matrix of $90^{\circ}$ is denoted by $P$ and that of glide is $Q$. Then $P=\begin{bmatrix}
0&-1\\1&0
\end{bmatrix}$ and $Q=\begin{bmatrix}
-1&0\\0&1
\end{bmatrix}$. By hypothesis hermite normal form of $PM=$ hermite normal form of $QM$. Therefore there exists an unimodular matrix $U=\begin{bmatrix}
p&q\\r&s
\end{bmatrix}$ such that $PM=QMU$. Putting expressions of $P, Q, U$ and comparing matrix entries we get $ap=b,\, aq=d,\, bp+rd=a,\, bq+sd=0$.

If $p=0$ then $b=0,\,rd=a,\, sd=0$. Since $d\neq 0$ so $s=0$. $a$ and $d$ are both positive so $rd=a \implies r>0$. Since, $U$ is unimodular so we have $rq=\pm 1$. Therefore, $r=1$. Hence, $a=d$. Therefore, $M=\begin{bmatrix}
a&0\\0&a
\end{bmatrix}$ which contradicts the fact that $M$ does not have $90^{\circ}$ rotation in its automorphism group. So, $p\neq 0$. Since $a\neq 0$ so $b\neq 0$. $bp+rd=a \implies 1-p^2=rq$ and $bq+sd=0 \implies p(p+s)=0 \implies s=-p$. For a fixed $d$  if we know $p$ then we can derive values of other unknowns. So number of such possible maps is given by $\sum_{\substack{d\mid n\\n \mid d^2}}[{\rm Number~ of ~solutions~ of} x^2-1\equiv 0 \pmod{\frac{d^2}{n}}]$(putting $q=d^2/n$ as $a=n/d)$. We will see in Claim \ref{cm-4}, this quantity is equals to $f_3(n)$. Since we are considering those maps which does not have $90^{\circ}$ and glide in its automorphism group so to get required maps we have to subtract $h(n)$ from above. 
Hence, $\alpha(n) = f_3(n)-h(n).$ This completes the proof of Claim \ref{alpha}.

\smallskip

Here, a map $Y$ will be $1$ orbital if and only if $\phi$ or glide reflection present in Aut$(Y)$ So whenever nither $\phi$ nor glide present Aut($Y$), $Y$ will be $2$-orbital.
Now, if both $\phi$ and glide present in Aut($X$) then we count it once in the collection of distinct $n$-sheeted maps.
If any one of $\phi$ or glide present in Aut($X$) we count it twice.
If none of them present then whenever for a map $Y$ after applying $\phi$ and glide we get equal maps we count it twice and otherwise we count it $4$ times.
Let $A(n)$ denotes number of $n$-sheeted maps up to isomorphism on which after applying $\phi$ and glide we get equal maps, and $B(n)$ denotes the number of maps up to isomorphism for which we get different maps. In both of these counting of $A(n)$ and $B(n)$ we considering those maps whose automorphism group does not contains $\phi$ and glide.
Then we have,
$$2A(n) + 4B(n) = \sigma(n) - f_2(n) - g(n) + h(n)~~~~~(*)$$
By Lemma \ref{alpha} we have $2A(n) = \alpha(n) \implies A(n) = \frac{\alpha(n)}{2}$.
Putting the value of $A(n)$ in $(*)$ we get,
$B(n) = \frac{1}{4} [\sigma(n) - f_2(n) - g(n) + h(n) - \alpha(n) ] $.
Thus number of $2$-uniform maps up to isomorphism is given by,
$$A(n) + B(n) = \frac{1}{4}[\sigma(n) - g(n) -f_2(n) + h(n) + \alpha(n)] =\frac{1}{4}[\sigma(n) - g(n) -f_2(n) + f_3(n)] .$$
Hence, total number of $n$-sheeted covers will be (last equality comes by using $\alpha(n)+h(n)=f_3(n)$),
$$ A(n)+B(n)+h(n)+\frac{g(n)-h(n)}{2}+\frac{f_2(n)-h(n)}{2}= \frac{1}{4}[\sigma(n)+f_2(n)+g(n)+f_3(n)].$$

Finally, suppose $Z=[4^1,8^2]$. Stabilizer of the tiling is $D_4$. For a map $X$ possible orders of $\mathcal{S}_X$ are 2, 4, 8 and we are counting it 4, 2, 1 times respectively. 
Let $\mathcal{Q}_i(n)$ denotes number of $n$-sheeted covers up to isomorphism with order of stabilizer $i$ for $i=2, 4, 8$.  Then we have,
\begin{equation}\label{eqn2}
    4\mathcal{Q}_2(n)+2\mathcal{Q}_4(n)+\mathcal{Q}_8(n) = \sigma(n)
\end{equation}
%Let $D_i'(n)$ be the number of distinct $n$-sheeted covers having $r_i'$ in its stabilizer fir $i=1,2$ and $D_3'(n)$ and $D_4'(n)$ denotes number of distinct $n$-sheeted covers having stabilizer generated by $90^{\circ}$ rotation and $90^{\circ}$ rotation, $r_1'$ respectively. 
To count number of distinct maps with this specific symmetries we make following claims.
Let $R_1'$ and $R_2'$ denotes the lines which bisects the angle between $OA_2$, $OB_2$ and $OB_2$, $O(-A_2)$ respectively. Let $r_i'$ be the map obtained by taking reflection of $E_2$ about $R_i'$ for $i=1,2$ respectively. Pictorially the lines are in Figure \ref{refl2}.
\begin{figure}[h]
    \centering
    \begin{tikzpicture}[scale=0.7]
	
		\node  (0) at (0, 0) {$\bullet$};
		\node  (1) at (-3, 0) {};
		\node  (2) at (3, 0) {};
		\node  (3) at (0, 3) {};
		\node  (4) at (0, -3) {};
		\node  (5) at (2.75, 2.75) {};
		\node  (6) at (-2.75, -2.75) {};
		\node  (7) at (-2.75, 2.75) {};
		\node  (8) at (2.75, -2.75) {};
		\node  (9) at (2, 0) {};
		\node  (10) at (0, 2) {};
		\node  (11) at (-2, 0) {};
		\node  (12) at (0, -2) {};
		\node  (13) at (2, 0) {$\bullet$};
		\node  (14) at (0, 2) {$\bullet$};
		\node  (15) at (-2, 0) {$\bullet$};
		\node  (16) at (0, -2) {$\bullet$};
		\node  (17) at (2, -0.5) {$A_2$};
		\node  (18) at (0.25, 1.75) {$~B_2$};
		\node  (19) at (-2, -0.25) {$-A_2$};
		\node  (20) at (0.25, -2.25) {$~~-B_2$};
		\node  (21) at (3, 2.75) {$R_1'$};
		\node  (22) at (-2.75, 3) {$R_2'$};
		\node  (23) at (0.25, -0.5) {$O$};

		\draw [style=dashed](6.center) to (5.center);
		\draw [style=dashed](7.center) to (8.center);
		\draw (1.center) to (2.center);
		\draw (3.center) to (4.center);
	
\end{tikzpicture}

    \caption{Line Reflections in $E_2$}
    \label{refl2}
\end{figure}
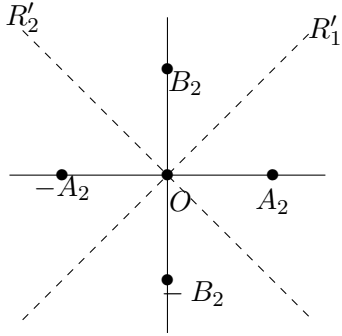
Let $\psi$ be the function obtained by taking $90^{\circ}$ rotation of $E_2$ about origin in anticlockwise direction.
\begin{claim}\label{cm-4}
If $X$ be a $n$-sheeted toroidal cover of a given map $X_0$, then the total number of maps having $r_i'$, for $i=1,2$, in the automorphism group of $X$ is equal to $f_3(n)$ as defined in Claim \ref{cm--3}.
\end{claim}
\noindent \emph{Proof of Claim \ref{cm-4} :}
Let $f_4(n)$ denotes the total number of maps having $r_1'$, in the automorphism group of $X$.
Observe that $r_1'$ sends $A_2$ to $B_2$ and $B_2$ to $A_2$. Proceeding in similar way as in Claim \ref{cm-1} we the polynomial for $r_1'$ will be $x^2-1$. 
$\rho_{0,-1}(1)=1$, $\rho_{0,-1}(2)=1$, $\rho_3(4)=2$ and $\rho_3(2^k)=4$ for all $k\in\mathbb{N}$ with $k\geq3$.
$\rho_{0,-1}(1)=1$ since $x=0$ is the solution for $x^2-1\equiv0\pmod1$.
$\rho_{0,-1}(2)=1$ since $1$ is the only solution of $0$ in $\mathbb{Z}_2$.
$\rho_3(4)=2$ since $1$ and $3$ are the solutions of $x^2-1=0$ in $\mathbb{Z}_4$.\\
Now, $x^2-1\equiv 0\pmod{2^k}\implies x^2\equiv1\pmod{2^k}$. We know that the congruence $x^2\equiv1\pmod{2^k}$ has exactly four incongruent solutions for $k\geq 3$. Hence, $\rho_{0,-1}(2^k)=4$. 

Proceeding in similar way as in the proof of Claims \ref{cm-1} and \ref{cm---2} we get
$f_4(2^k)=2k-1$ when $k\in\mathbb{N}$ and for odd prime $p$,
$f_4(p^k)=k+1$ where $k\in\mathbb{N}$.
and for $n=2^{k_0}p_1^{k_1}p_2^{k_2}\dots p_{n_1}^{k_{n_1}}$ where $p_i$ is any odd prime for $i\in \{0,1,\dots,n_1\}$ we get total number of maps having $r_1'\in$ Aut($X$) is,
\[ f_4(n)=\left\{
	\begin{array}{ll}
		\prod_{i=1}^{n_1}(k_i+1)  & \mbox{, if } k_0=0 \\
		(2k_0-1)\prod_{i=1}^{n_1}(k_i+1) & \mbox{, otherwise.}
	\end{array}
\right. \]
Observe that $f_4(n) = f_3(n)$. Similarly, for $r_2'$ we will get the same function. This completes the proof of Claim \ref{cm-4}.

\begin{claim}\label{lem-2}
If $X$ be a $n$-sheeted toroidal cover of a given map $X_0$, then the total number of maps having $\psi$ in the automorphism group of $X$ is equal to $f_2(n)$. Where $\psi$ denotes the map obtained by taking $90^{\circ}$ rotation of $E_2$ with respect to origin.
\end{claim}
\noindent\emph{Proof of Claim \ref{lem-2} :}
Since $\psi$ is $90^{\circ}$ rotation of $E_2$ so the proof is exactly same as of Claim \ref{cm---2}.

\begin{claim}\label{cm-5}
Let $X$ be an $n$-sheeted cover such that $\mathcal{S}_X=\mathcal{S}_{E_2}$. Then number of such possible $X$ is given by $$f_6(n):=\left\{
	\begin{array}{ll}
		0  & \mbox{, if } k_i\equiv1\pmod2 \mbox{ for some } i\in\{0,1,2,\dots,n_1\}\\
		1 & \mbox{, otherwise.}
	\end{array}
\right.$$ 
where 
$n=2^{m}\cdot 3^{k_0} \cdot \prod_{i=1}^{n_1}p_i^{k_i}$ such that $p_i$ is any prime other than $2$ for $i\in \{0,1,\dots,n_1\}$.
\end{claim}
\noindent\emph{Proof of Claim \ref{cm-5} :} 
Since, $\mathcal{S}_X=\mathcal{S}_{E_2}$ so Aut($X$) will contain $\psi$ and $r_1'$. Thus, the total number of such maps is given by the number of solutions of gcd$(x^2+1,\,1-x^2)$ in $\mathbb{Z}_{\frac{d^2}{n}}$ satisfying $n\mid d^2$ for every divisor $d$ of $n$.
Let $\rho_6(n):=\#\{x\in \mathbb{Z}_n:\, x^2+1=0\,{\rm and}\,1-x^2=0\}$.
Therefore
$$f_6(n):=\sum_{\substack{d \mid n \\n \mid d^2}}\rho_6\left(\frac{d^2}{n}\right).$$ 
By similar reason $\rho_6$ is multiplicative and as a consequence we have $f_6$ is also multiplicative. Thus it is enough to calculate for prime powers.
$$f_6(p^k) =\left\{
	\begin{array}{ll}
		\sum_{l= \frac{k}{2} }^k\rho_{6}\left(p^{2l-k}\right)  & \mbox{, if  k is even} \\
		\sum_{l= \frac{k+1}{2} }^k\rho_{6}\left(p^{2l-k}\right) & \mbox{, if  k is odd}
	\end{array}
          \right.$$
Now, $\rho_{6}(1)=1$ since $x=0$ is the solution for the congruences $x^2+1\equiv0\pmod1$ and $x^2-1\equiv0\pmod1$.
Since $x^2+1=x^2-1$ in $\mathbb{Z}_2$, gcd$\{(x^2+1),\,(x^2-1)\}=x^2+1$.
Since $x=1$ is the only solution of $x^2+1$ in $\mathbb{Z}_2$, $\rho_{6}(2)=1$.

Since gcd$\{(x^2+1),\,(x^2-1)\}=1$ in $\mathbb{Z}_{2^k}$ where $k\in \mathbb{N}\setminus\{1\}$, we have $\rho_6(2^k)=0$ for all $k$.\\
$$f_6(2^k) =\left\{
	\begin{array}{ll}
		\sum_{l= \frac{k}{2} }^k\rho_{6}\left(2^{2l-k}\right) = 1 & \mbox{, if  k is even} \\
		\sum_{l= \frac{k+1}{2} }^k\rho_{6}\left(2^{2l-k}\right) = 1  & \mbox{, if  k is odd}
	\end{array}
          \right. $$

Since gcd$\{(x^2+1),\,(x^2-1)\}=1$ in $\mathbb{Z}_{p^k}$, we have $\rho_3(p^k)=0$ for all $k\in \mathbb{N}$. Thus we have,
Let $n=p^k$. Then, $d=p^i$, $i\leq k$.
$$f_6(p^k) =\left\{
	\begin{array}{ll}
		\sum_{l= \frac{k}{2} }^k\rho_{6}\left(2^{2l-k}\right) = 1 & \mbox{, if  k is even} \\
		\sum_{l= \frac{k+1}{2} }^k\rho_{6}\left(2^{2l-k}\right) = 0  & \mbox{, if  k is odd}
	\end{array}
          \right. $$

Let $n=2^{m}3^{k_0}p_1^{k_1}p_2^{k_2}\dots p_{n_1}^{k_{n_1}}$ where $p_i$ is any prime other than $2$ for $i\in \{0,1,\dots,n_1\}$.\\
Then, $$f_6(n)=\left\{
	\begin{array}{ll}
		0  & \mbox{, if } k_i\equiv1\pmod2 \mbox{ for some } i\in\{0,1,2,\dots,n_1\}\\
		1 & \mbox{, otherwise.}
	\end{array}
\right.$$ 
This completes the proof of Claim \ref{cm-5}.

\smallskip

Now, clearly, $\mathcal{Q}_8(n) = f_6(n)$. The stabilizer of elements of $\mathcal{Q}_4(n)$ are either generated by $90^{\circ}$ rotation or point reflection with $r_1'$ or point reflection with $r_2'$. Hence, $\mathcal{Q}_4(n)=(f_4(n)-f_6(n))+(f_4(n)-f_6(n))+(f_2(n)-f_6(n))=f_2(n)+2f_4(n)-3f_6(n)$.
$\mathcal{Q}_2(n)$ can be determined uniquely from equation (\ref{eqn2}). Thus, total number of $n$-sheeted covers up to isomorphism is $\mathcal{Q}_2(n)+\mathcal{Q}_4(n)+\mathcal{Q}_8(n)= \frac{1}{4}[\sigma(n)+2f_2(n)+4f_4(n)-3f_6(n)]$.
\end{proof}

Finally, we prove Theorem \ref{thm-main4} and show existence of $k$-orbital minimal covers.

\begin{proof}[Proof of Theorem \ref{thm-main4}]
Let $X$ be a $m$-orbital map of vertex type $[p_1^{n_1}, \dots p_k^{n_k}]$ and $Y_1$ be a $k (\le m)$-orbital cover of $X$. Let number of sheets of the cover $Y_1\longrightarrow X$ be $n_1$. Consider the set $C_1$ containing all $s$ sheeted covering of $X$ for $s\le n_1-1$. Now, we check the existence of a $k$-orbital cover in $C_1$. If it does not exist, then $Y_1$ be a minimal $k$-orbital cover otherwise take $Y_2$ be a $k$-orbital cover in $C_1$. Let number of sheets for the covering $Y_2\longrightarrow X$ be $n_2$. Then consider $C_2$ be the collection of all $s$ sheeted cover of $X$ for $s\le n_2$. Again check existence of a $k$-orbital cover in $C_2$. If not then $Y_2$ be a minimal $k$-orbital cover of $X$ otherwise proceed similarly to more lower sheeted covering. As there are only finitely many covers for each sheet so the process will terminate after finite number of steps. This completes the proof of Theorem \ref{thm-main4}.
\end{proof}

\section{Acknowledgements}

Authors are supported by NBHM, DAE (No. 02011/9/2021-NBHM(R.P.)/R$\&$D-II/9101).

%Arnab Kundu is supported by NBHM, DAE. Dr. Dipendu Maity is supported by NBHM, DAE (No. 02011/9/2021-NBHM(R.P.)/R$\&$D-II/9101).

{\small
%% BIBLIOGRAPHY %%

}


\begin{thebibliography}{99}

\bibitem{A1973}
A. Altshuler, {\em Construction and enumeration of regular maps on the torus}, Discrete Math. {\bf 4} (1973) 201-217.

\bibitem{babai:1991}
L. Babai, {\em Vertex-transitive graphs and vertex-transitive maps}, J. Graph Theory {\bf 15} (1991) 587–627.

\bibitem{brehm:2008}
U. Brehm and W. K\"{u}hnel, {\em Equivelar maps on the torus}, European J. Combin. {\bf 29} (2008) 1843-1861.

\bibitem{brehm:1997}
U. Brehm and E. Schulte, Polyhedral maps, Handbook of Discrete and Computational Geometry (eds. J. E. Goodman and J. O' Rourke), CRC Press NY, 1997, 345-358.
%
\bibitem{burton2018}
D. M. Burton, {\em Elementary Number Theory}, Seventh Edition, Mc Graw Hill, 2018

\bibitem{CM1957} H. S. M. Coxeter and W. O. J. Moser, {\em Generators and Relations for Discrete Groups}, 4th edn., Springer, Berlin, 1980.
%

\bibitem{DKM2022}
L. Das, A. Kundu and D. Maity, {\em Enumeration of 2-uniform maps on the torus}, arXiv:2206.11144.

\bibitem{BD2020}
B. Datta, {\em Vertex-transitive covers of semi-equivelar toroidal maps}, arXiv:2004.09953.

\bibitem{DM2017}
B. Datta and D. Maity, {\em Semi-equivelar and vertex-transitive maps on the torus}, Beitr\"{a}ge Algebra Geom. {\bf 58} (2017) 617-634.
    
\bibitem{DM2018}
B. Datta and D. Maity, {\em Semi-equivelar maps on the torus and the Klein bottle are Archimedean},  Discrete Math. {\bf 341} (2018) 3296-3309.

\bibitem{DU2005}
B. Datta and A. K. Upadhyay, {\em Degree-regular triangulations of torus and Klein bottle}, Proc. Indian Acad. Sci. Math. Sci. {\bf 115} (2005) 279-307.
%
\bibitem{drach:2015}
K. Drach and M. Mixer, {\em Minimal covers of equivelar toroidal maps}, Ars Math. Contemp. {\bf 9} (2015) 77-91.
%
\bibitem{drach:2019}
K. Drach, Y. Haidamaka, M. Mixer and M. Skoryk, {\em Archimedean toroidal maps and their minimal almost regular covers}, Ars Math. Contemp. {\bf 17} (2019) 493-514.

\bibitem{FT1965} L. Fejes T\'{o}th, {\em Regul\"{a}re Figuren}, Akad\'{e}miai Kiad\'{o}, Budapest, 1965. English translation: {\em Regular Figures}, Pergmon Press, Oxford,  1964.
%

\bibitem{GS2016}
B. Gr\"{u}nbaum and  G. C. Shephard, {\em Tilings and Patterns}, Dover Publication, Inc. Mineola, New York.

\bibitem{GS1977}
B. Gr\"{u}nbaum and  G. C. Shephard, {\em Tilings by regular polygons: Patterns in the plane from Kepler to the present, including recent results and unsolved problems}, Math. Mag., {\bf 50} (1977), pp. 227-247.
%
\bibitem{HW1979}
G. H. Hardy and E. M. Wright,  {\em An introduction to the theory of numbers}, Sixth edition. The Clarendon Press, Oxford University Press, New York, 1979.
%
\bibitem{HW2012}
M. I. Hartley, D. Pellicer and G. Williams,  {\em Minimal covers of the prisms and antiprisms},  Discrete Math. {\bf 312} (2012) 3046-3058.
%
\bibitem{HW1979}
G. H. Hardy and E. M. Wright,  {\em An introduction to the theory of numbers}, Sixth edition. The Clarendon Press, Oxford University Press, New York, 1979.
%
\bibitem{AH2002}
A. Hatcher, {\em Algebraic topology}, Cambridge University Press, Cambridge, 2002.
%
\bibitem{hubard:2012}
I. Hubard, A. Orbani\'{c}, D. Pellicer and A. I. Weiss, {\em Symmetries of equivelar 4-toroids}, Discrete Comput. Geom. {\bf 48} (2012) 1110-1136.

\bibitem{jones:1978}
G. A. Jones and D. Singerman, {\em Theory of maps on orientable surfaces}, Proc. London Math. Soc. {\bf 37} (1978) 273-307.

\bibitem{katok:1992}
S. Katok, {\em Fuchsian groups}, Chicago Lectures in Mathematics, University of Chicago Press, Chicago, 1992.

\bibitem{kurth:1986}
W. Kurth, {\em Enumeration of Platonic maps on the torus}, Discrete Math. {\bf 61} (1986) 71-83.

\bibitem{MU2018}
D. Maity and A. K. Upadhyay, {\em On the enumeration of a class of toroidal graphs},  Contrib. Discrete Math. {\bf 13} (2018) 79-119.

\bibitem{MPW2013}
M. Mixer, D. Pellicer and G. Williams, {\em Minimal covers of the Archimedean tilings, part II},  Electron. J. Combin. {\bf 20} (2013), Paper 20, 19 pp.

\bibitem{negami:1983}
S. Negami, {\em Uniqueness and faithfulness of embedding of toroidal graphs}, Discrete Math. {\bf 44} (1983) 161-180.

\bibitem{pel2011}
D. Pellicer and A. I. Weiss, {\em Uniform maps on surfaces of non-negative Euler characteristic}, Symmetry: Culture and Science {\bf 22} (2011) 159-966.

\bibitem{pw2011}
D. Pellicer and G. Williams, {\em Minimal covers of the Archimedean tilings, part 1}, Electron. J. Combin.{\bf 19} (2012), Paper 6, 37 pp.

\bibitem{su:vt11}
O. \v{S}uch, {\em Vertex-transitive maps on a torus}, Acta Mathematica Universitatis Comenianae, {\bf 53} (2011) 1-30.


\end{thebibliography}
\end{document}